\documentclass{amsart}

\usepackage{amsthm,amsfonts,amsmath,amssymb,latexsym,epsfig,mathrsfs,yfonts,marvosym}
\usepackage[usenames]{color}

\DeclareMathAlphabet\oldmathcal{OMS}        {cmsy}{b}{n}
\SetMathAlphabet    \oldmathcal{normal}{OMS}{cmsy}{m}{n}
\DeclareMathAlphabet\oldmathbcal{OMS}       {cmsy}{b}{n}
 
\usepackage{eucal}

\usepackage{graphicx}
\usepackage[all]{xy}
\usepackage{epsfig}

\newtheorem{theorem}{Theorem}[section]

\newtheorem{proposition}[theorem]{Proposition}
\newtheorem{corollary}[theorem]{Corollary}
\newtheorem{definition}[theorem]{Definition}
\newtheorem{question}{Question}
\newtheorem{ack}{Acknowledgments} 
\newenvironment{example}{\medskip \refstepcounter{theorem}
\noindent  {\bf Example \thetheorem}.\rm}{\,}
\newenvironment{remark}{\medskip \refstepcounter{theorem}
\noindent  {\bf Remark \thetheorem}.\rm}{\,}
\newtheorem{conjecture}[theorem]{Conjecture}

\renewcommand{\thetheorem}{\thesection.\arabic{theorem}}
\def\d{\partial}

\def\<{\langle}

\def\>{\rangle}

\def\BOne{{\mathchoice {\rm 1\mskip-4mu l} {\rm 1\mskip-4mu l}
                          {\rm 1\mskip-4.5mu l} {\rm 1\mskip-5mu l}}}
\def\Ric{{\rm Ric}}

\def\fract#1#2{\raise4pt\hbox{$ #1 \atop #2 $}}
\def\decdnar#1{\phantom{\hbox{$\scriptstyle{#1}$}}
\left\downarrow\vbox{\vskip15pt\hbox{$\scriptstyle{#1}$}}\right.}

\def\bbc{{\mathbb C}}

\def\bbn{{\mathbb N}}

\def\bbp{{\mathbb P}}
\def\bbq{{\mathbb Q}}
\def\bbr{{\mathbb R}}

\def\bbz{{\mathbb Z}}

\def\gra{\alpha}

\def\grd{\delta}

\def\grg{\gamma}
\def\gri{\iota}
\def\grk{\kappa}
\def\grl{\lambda}

\def\gro{\omega}

\def\grr{\rho}

\def\grt{\tau}
\def\gru{\upsilon}

\def\grz{\zeta}

\def\grG{\Gamma}

\def\grL{\Lambda}
\def\grO{\Omega}

\def\grS{\Sigma}

\def\bfd{{\bf d}}

\def\bff{{\bf f}}

\def\bfl{{\bf l}}

\def\bfw{{\bf w}}

\def\bfz{{\bf z}}

\def\calc{{\mathcal C}}
\def\calo{{\mathcal O}}

\def\cald{{\mathcal D}}
\def\cale{{\mathcal E}}
\def\calf{{\mathcal F}}

\def\calh{{\mathcal H}}

\def\calk{{\mathcal K}}

\def\calm{{\mathcal M}}

\def\calo{{\mathcal O}}

\def\calr{{\mathcal R}}
\def\cals{{\oldmathcal S}}

\def\calv{{\mathcal V}}
\def\calw{{\mathcal W}}

\def\calz{{\oldmathcal Z}}

\def\Se{Sasaki-Einstein }
\def\la#1{\hbox to #1pc{\leftarrowfill}}
\def\ra#1{\hbox to #1pc{\rightarrowfill}}

\def\ga{{\mathfrak a}}

\def\gc{{\mathfrak c}}

\def\gh{{\mathfrak h}}

\def\gm{{\mathfrak m}}

\def\go{{\mathfrak o}}

\def\gr{{\mathfrak r}}
\def\gs{{\mathfrak s}}
\def\gt{{\mathfrak t}}
\def\gu{{\mathfrak u}}

\def\gA{{\mathfrak A}}

\def\gC{{\mathfrak C}}

\def\gF{{\mathfrak F}}
\def\gG{{\mathfrak G}}

\def\gI{{\mathfrak I}}

\def\gR{{\mathfrak R}}
\def\gS{{\mathfrak S}}

\def\X{\frak{X}}

\def\hook{\mathbin{\hbox to 6pt{%
                 \vrule height0.4pt width5pt depth0pt
                 \kern-.4pt
                 \vrule height6pt width0.4pt depth0pt\hss}}}

\begin{document}
\bibliographystyle{amsalpha}

\title{Sasakian Geometry: the recent work of Krzysztof Galicki}\thanks{During the preparation of this work the author was partially supported by NSF grant DMS-0504367.}

\author{Charles P. Boyer}
\address{Department of Mathematics and Statistics,
University of New Mexico, Albuquerque, NM 87131.}

\email{cboyer@math.unm.edu} 

\maketitle

\centerline{Dedicated to the memory of Kris Galicki}
\bigskip
%\begin{abstract}

%\end{abstract}
\tableofcontents

\section{Introduction}

As many know, on July 8, 2007 Kris Galicki met with a tragic accident while hiking in the Swiss Alps. He was flown back home to Albuquerque, New Mexico, USA, where he remained in a coma until his death on September 24, 2007. Kris' last conference talk was an invited talk at the conference to which these proceedings are dedicated. I have been asked by the editors, Rosanna Marinosci and Domenico Perrone, to provide a paper to the proceedings in Kris' honor. I am honored to accept this invitation, but I turn to it with a deep sense of remorse over the loss of a dear friend and colleague. Kris and I had a wonderful and mutually satisfying working relationship, because we very much enjoyed exchanging ideas and insights, and developing the mathematics that follows. Our most recent completed project is our book, Sasakian Geometry \cite{BG05}, and it saddens me immensely that Kris will never see the finished product. 

The paper is organized as follows: In Section \ref{revsas} I give a brief review of Sasakian geometry emphasizing its relation with various other structures, such as contact structures, CR structures, foliations with transverse K\"ahler structures, and conical structures. By in large this material was taken from our book \cite{BG05}, although I do give a different focus on certain things. Chapter 3 is devoted to studying the influence of isolated conical singularities on Sasakian geometry. This method has proved to be particularly productive in the special case of isolated hypersurface singularities arising from weighted homogeneous polynomials. 

 In Section 4 I describe Sasakian geometry in low dimensions. Of particular interest is dimension 5, where Kris and I in collaboration with others have given many examples of Sasaki-Einstein metrics emphasizing the impact of Sasakian geometry on the understanding of Einstein manifolds. Worthy of special mention is the foundational work of Koll\'ar on 5-dimensional Seifert structures and their relation to Sasakian structures. I also take this opportunity to present some new results involving Sasaki-Einstein metrics in dimension 7 which arose out of email exchanges between Kris and myself in June 2007.

Finally, in Section 5, I present mainly the work of others, namely, that of Martelli, Sparks, and Yau, and Futaki, Ono, and Wang regarding toric Sasaki-Einstein structures. Kris was very interested in the very important recent result of Futaki, Ono, and Wang which proves the existence of a Sasaki-Einstein metric within the Sasaki cone of any toric contact manifold of Sasaki type with vanishing first Chern class $c_1(\cald)$ of the contact bundle. We had planned to give a classification of such structures.

\begin{ack}
I would like to thank E. Lerman for discussions involving holomorphic and symplectic fillability, and P. Massot for emails that pointed out some inaccuracies in my presentation of 3-manifold contact topology in earlier versions of this paper, and also gave me to some important references. Of course, any errors or omissions are mine and mine alone. I would also like to thank Evan Thomas for the computer programs that he provided for both Kris and me.
\end{ack}

\section{A Review of Sasakian geometry}\label{revsas}

I begin with a brief description of Sasakian geometry. A detailed discussion of both Sasakian and contact geometry can be found in \cite{BG05}. A Sasakian structure incorporates several well-known geometries, contact geometry with a chosen contact 1-form $\eta,$ CR geometry with a strictly pseudoconvex Levi form, and a 1-dimensional foliation with a transverse K\"ahler structure.

\subsection{Almost Contact and Contact Structures}
A triple $(\xi, \eta, \Phi)$ defines an {\it almost contact structure}
on $M$ if  $\xi$ is a nowhere vanishing vector field,
$\eta$ is a one form, and $\Phi$ is a tensor of type $(1,1)$, such that
\begin{equation}\label{almostcon}
\eta(\xi) = 1 \, , \quad \Phi^2 = -\BOne + \xi \otimes \eta \, .
\end{equation}
The vector field $\xi$ defines the {\it characteristic foliation} 
${\mathcal F}_{\xi}$ with one-dimensional leaves, and the kernel of $\eta$
defines the codimension one sub-bundle ${\mathcal D}=\ker~\eta$. We have the
canonical splitting 
 \begin{equation}
 TM = {\mathcal D} \oplus L_{\xi}\, , \label{cs}
 \end{equation}
where $L_\xi$ is the trivial line bundle generated by $\xi.$
If the 1-form satisfies 
\begin{equation}\label{constr}
\eta\wedge(d\eta)^n\neq 0,
\end{equation} 
then the subbundle $\cald$ defines a contact structure on $M.$ In this case the vector field $\xi$ is called the {\it Reeb vector field}. For every choice of 1-form $\eta$ in the underlying contact structure $\cald,$ this vector field is unique. If $\eta$ defines a contact structure, then the triple $(\xi,\eta,\Phi)$ is said to be a {\it K-contact} structure, if $\pounds_\xi\Phi=0.$ Furthermore, for any contact manifold $(M,\cald)$, we will fix an orientation on $M$ as well as a co-orientation on $\cald.$
Let $\cald_{ann}$ denote the annihilator of $\cald$ in $\grL^1(M)$. Then, $\cald_{ann}$ minus the $0$-section splits as $\cald_{ann}^+\cup \cald_{ann}^-.$ Fixing a section $\eta_0$ of $\cald_{ann}^+$ we can identify the sections $\grG(\cald_{ann}^+)$ with the set $\{\eta ~|~\eta=f\eta_0 ~\text{for some positive function $f$}\}.$ We define a map $\grG(\cald_{ann}^+)\ra{1.5} \X(M)$, where $\X(M)$ denotes the Lie algebra of vector fields on $M,$ by associating to $\eta\in \grG(\cald_{ann}^+)$ its Reeb vector field $\xi,$ and we denote by $\calr^+(\cald)$ its image in $\X(M)$. This map is clearly injective, and we give $\calr^+(\cald)$ the subspace topology as a subspace of the space of smooth sections of the vector bundle $T(M)$ with the $C^\infty$ compact-open topology.

\begin{proposition}\label{Reebvf}
The subspace $\calr^+(\cald)$ is an open convex cone in $\X(M).$
\end{proposition}

\begin{proof}
For $i=1,2$ let $\xi_i\in \calr^+(\cald)$ be the Reeb vector fields for $\eta_i\in \cald_{ann}^+$ with a reference 1-form $\eta_0$. Then 
$$\eta_0(t_1\xi_1 +t_2\xi_2)=t_1\eta_0(\xi_1)+t_2\eta_0(\xi_2)>0$$
for $t_1,t_2\geq 0,$ and not both $0.$ Set $\xi_{12}=t_1\xi_1 +t_2\xi_2,$ and define $\eta_{12}=\frac{1}{\eta_0(\xi_{12})}\eta_0.$
Then $\eta_{12}(\xi_{12})=1,$ and $\xi_{12}\hook d\eta_{12}=\pounds_{\xi_{12}}\eta_{12}.$ Since $\xi_i$ are Reeb vector fields, they leave sections of the contact bundle $\cald$ invariant. So $\xi_{12}=t_1\xi_1 +t_2\xi_2$ also leaves $\cald$ invariant. Thus, $\pounds_{\xi_{12}}\eta_{12}=g\eta_{12}$ for some smooth function $g.$ Evaluating this on $\xi_{12}$ gives $g=(\pounds_{\xi_{12}}\eta_{12})(\xi_{12})=\xi_{12}\bigl(\eta_{12}(\xi_{12})\bigr)=0$ which implies that $\xi_{12}$ is the Reeb vector field of $\eta_{12}.$ That $\calr^+(\cald)$ is open follows immediately from positivity.
\end{proof}

\subsection{Almost CR and CR Structures}
Given an almost contact structure, the tensor field $\Phi$ restricted to the subbundle $\cald$ defines an almost complex structure $J,$ so $(\cald,J=\Phi|_\cald)$ defines a codimension one {\it almost CR structure}\footnote{For us both an almost CR and CR structure will always be of codimension one in which case they are often referred to as (almost) CR structures of {\it hypersurface type}.}.
The condition of integrability of the almost CR structure can be phrased in a variety of ways (cf. \cite{BG05} and references therein); however, from the point of view of CR geometry the most cogent is this. The almost contact structure defines a splitting of the complexification $\cald\otimes \bbc$ as 
$$\cald\otimes \bbc =\cald^{1,0}\oplus \cald^{0,1}$$
where the bundles $\cald^{1,0}$ and $\cald^{0,1}$ correspond to the eigenvalues $+i$ and $-i$, of $J$ respectively, and satisfy 
$$\cald^{1,0}\cap \cald^{0,1}=\{0\}.$$
Let $\grG(E)$ denote the vector space of smooth sections of the vector bundle $E.$ Then the almost CR structure $(\cald,J)$ is said to be {\it formally integrable} if the condition
$$[\grG(\cald^{1,0}),\grG(\cald^{1,0})]\subset \grG(\cald^{1,0})$$
holds. A manifold with a formally integrable CR structure is often called an {\it abstract CR manifold}. An almost CR manifold $M^{2n+1}$ of dimension $2n+1$ is {\it integrable} if about each point there are `holomorphic coordinates', that is complex valued functions $z_1,\ldots,z_n$ such that the differentials $\{dz_i\}_{i=1}^n$ are linearly independent and the functions $z_i$ are annihilated by any section of $\cald^{0,1}.$ Of much interest in CR geometry is the so-called {\it local embeddability problem} \cite{BER99} which amounts to a formally integrable CR structure being integrable. However, for us this problem is moot.

We are interested in the case that the {\it Levi form} $L_\eta :=d\eta \circ (J\otimes \BOne)$ is positive or negative definite. In this case the almost CR structure is said to be {\it strictly pseudoconvex}, and the 1-form $\eta$ defines a {\it contact structure} on $M.$ A deep theorem of Kuranishi \cite{Kur82}  says that any formally integrable strictly pseudoconvex CR manifold of dimension 9 or greater is integrable. This result is known to be false in dimension 3 \cite{BER99}, and as far as I know is still open in dimension 5. It was shown to hold in dimension 7 by Akahori \cite{Aka87}. However, we are mainly interested in  K-contact structures, that is, the Reeb vector field $\xi$ is an infinitesimal CR transformation (see Definition \ref{kconsas} below). In this case it is easy to see that any formally integrable K-contact structure is integrable.

\subsection{Compatible Riemannian Metrics}
Given the triple $(\xi, \eta, \Phi)$, one can then ask for a compatible Riemannian metric $g$ in the sense that 
$$g(\Phi(X), \Phi(Y))=g(X,Y)-\eta (X)\eta(Y) \, .$$ 
Such metrics always exist and then the quadruple $\cals=(\xi,\eta,\Phi,g)$ is called an {\it almost contact metric structure}, and a {\it contact metric structure} in the case that $(\xi, \eta, \Phi)$ defines a contact structure. In the latter case, with $\xi$ and $\eta$ fixed there is a 1-1 correspondence between such compatible Riemannian metrics and $(1,1)$ tensor fields $\Phi$, or alternatively, almost complex structures on the symplectic vector bundle $(\cald,d\eta).$ So given the triple $(\xi,\eta,\Phi)$ satisfying Equations (\ref{almostcon}) and (\ref{constr}), the metric $g$ is given uniquely by
\begin{equation}\label{conmetric}
 g=g_\cald + \eta\otimes \eta =d\eta\circ (\Phi\otimes \BOne) +\eta\otimes \eta.
\end{equation}
On a given manifold we are interested in the set $\calc\calm(M)$ of all such contact metric structures. We give $\calc\calm(M)$ the $C^\infty$ topology as a subspace of the space of smooth sections of the corresponding vector bundles. We often suppress the $M$ and simply write $\calc\calm$ when the manifold is understood.

\begin{definition}\label{kconsas}
A contact metric structure $\cals=(\xi,\eta,\Phi,g)\in \calc\calm$ is said to be a {\bf K-contact} structure if $\pounds_\xi\Phi =0,$ or equivalently $\pounds_\xi g =0,$ and it is said to be a {\bf Sasakian} structure if it is K-contact and the almost CR structure $(\cald,J)$ is integrable.
\end{definition}

Note that for K-contact structures, the Reeb vector field $\xi$ is a Killing field, hence, the name K-contact.
One of the reasons that K-contact and Sasakian structures are tractable is that the characteristic foliation is {\it Riemannian}, or equivalently {\it bundle-like} \cite{Mol88}. 
A (K-contact) Sasakian structure $\cals=(\xi,\eta,\Phi,g)$ determines a strictly pseudoconvex (almost) CR structure $(\cald,J)$ by setting $\cald=\ker~\eta,$ and $J=\Phi|_\cald$. Conversely, given a strictly pseudoconvex (almost) CR structure $(\cald,J)$, we can inquire about the space of (K-contact) Sasakian structures that have $(\cald,J)$ as its underlying (almost) CR structure. A necessary and sufficient condition that a contact metric structure $\cals=(\xi,\eta,\Phi,g)$ be K-contact is that $\xi$ belong to the Lie algebra $\gc\gr(\cald,J)$ of infinitesimal almost CR transformations. The subspace $\gc\gr^+(\cald,J)=\calr^+(\cald)\cap \gc\gr(\cald,J)$ of all Reeb vector fields belonging to a (K-contact) Sasakian structure with underlying (almost) CR structure $(\cald,J)$ is invariant under the group $\gC\gR(\cald,J)$ of CR transformations, and the quotient space $\grk(\cald,J)=\gc\gr^+(\cald,J)/\gC\gR(\cald,J)$ is a cone called the {\it Sasaki cone} \cite{BGS06}. Let $\gt_k$ be the Lie algebra of a maximal torus $T_k$ in $\gC\gR(\cald,J),$ and let $\calw$ denote the Weyl group of a maximal compact subgroup of $\gC\gR(\cald,J).$ Then we can identify the Sasaki cone $\grk(\cald,J)$ with quotient $\gt_k^+/\calw,$ where $\gt_k^+=\gt_k\cap \gc\gr^+(\cald,J).$ By abuse of terminology we often refer to $\gt_k^+$ as the Sasaki cone for convenience. Moreover, it is well-known that the dimension $k$ of the Sasaki cone takes values $1\leq k\leq n+1.$ When $k=n+1$ we are in the realm of toric Sasakian geometry \cite{BG00b,Ler02a} to be discussed briefly later.

It is known that $\gC\gR(\cald,J)$ is a Lie group and there are only two cases when it is non-compact \cite{Lee96,Sch95}, when $M$ is the Heisenberg group with its standard CR structure, or when $M=S^{2n+1}$ with its standard CR structure. In the latter case $\gC\gR(\cald,J)=SU(n+1,1)$ \cite{Web77}. Moreover, if $M$ is compact and $(\cald,J)$ is of Sasaki type, but is not the standard CR structure on $S^{2n+1}$ then there is a Sasakian structure $\cals\in {\mathcal S}(\cald,J)$ such that $\gC\gR(\cald,J)=\gA\gu\gt(\cals),$ the automorphism group of $\cals$ \cite{BGS06}.

We now define various subspaces of $\calc\calm$ of interest.

\begin{definition}\label{cmsub}
If $\calc\calm$ denotes the space of all contact metric structures on $M,$ then we define
\begin{enumerate}
\item $\calk$, the subspace consisting of all K-contact structures in $\calc\calm;$
\item ${\mathcal S},$ the subspace consisting of all Sasakian structures in $\calc\calm;$
\item $\calc\calm(\cald)$, the subspace of all contact metric structures with underlying contact structure $\cald;$
\item $\calc\calm(\cald,J)$, the subspace of all contact metric structures with underlying CR structure $(\cald,J).$
\item $\calc\calm(\calf_\xi)$ the subspace of all contact metric structures whose Reeb vector field belongs to the foliation $\calf_\xi$ for a fixed Reeb vector field $\xi.$
\item The corresponding definitions in items (3)-(5) with $\calc\calm$ replaced by $\calk$ or ${\mathcal S}.$
\end{enumerate}
\end{definition}

There are obvious inclusions ${\mathcal S}(\cald)\subset \calk(\cald)\subset \calc\calm(\cald),$ etc.

\subsection{The Characteristic Foliation and Basic Cohomology}
The interplay between the contact/CR structure point of view on the one hand and the characteristic foliation point of view on the other will play an important role for us. Here we concentrate on the latter.

Recall that given a foliation $\calf$ on a manifold $M$ an $r$-form $\gra$ is {\it basic} if for all vector fields $V$ that are tangent to the leaves of $\calf$ the following conditions hold:
\begin{equation}\label{basic}
V\hook \gra=0 \qquad \pounds_V\gra=0.
\end{equation}
So on a foliated manifold $M$ we can consider the subalgebra $\grO^r_B(\calf)\subset \grO^r(M)$ of basic differential $r$-forms on $M.$ It is easy to see that exterior differentiation takes basic forms to basic forms, and gives rise to a de Rham cohomology theory, called {\it basic cohomology} whose groups are denoted by $H^r_B(\calf),$ and their cohomology classes by $[\gra]_B.$ These groups applied to the characteristic foliation play an important role in Sasakian and K-contact geometry \cite{BG05}. Indeed, for a K-contact manifold we have the exact sequence 
\begin{equation}\label{bascohseq}
\cdots\ra{2.1}H_B^p(\calf_\xi)\fract{\gri_*}{\ra{2.1}}H^p(M,\bbr)\fract{j_p}{\ra{2.1}}
H_B^{p-1}(\calf_\xi) \fract{\grd}{\ra{2.1}}
H^{p+1}_B(\calf_\xi)\ra{2.1}\cdots\, ,
\end{equation}
where $\grd$ is the connecting homomorphism given by
$\grd[\gra]_B=[d\eta\wedge \gra]_B=[d\eta]_B\cup[\gra]_B,$ and
$j_p$ is the map induced by $\xi\hook.$ In particular, the beginning of the sequence gives two important pieces of the puzzle, namely
\begin{equation}\label{bascohseq2}
 H^1_B(\calf_\xi)\approx H^1(M,\bbr), \qquad 0\ra{1.0}\bbr \fract{\grd}{\ra{1.0}} H^2(\calf_\xi) \fract{\gri_*}{\ra{1.0}} H^2(M,\bbr)\ra{1.0} \cdots .
\end{equation}
The closed basic 2-form $d\eta$ defines a non-trivial class $[d\eta]_B\in H^2(\calf_\xi)$ which is zero in $H^2(M,\bbr).$ In fact, the kernel of $\gri_*$ at level 2 is generated by $[d\eta]_B.$ Another important basic cohomology class is the {\it basic first Chern class} $c_1(\calf_\xi),$ the existence of which is due to the fact that on a K-contact manifold the transverse Ricci tensor is basic. Note that $\gri_*c_1(\calf_\xi)=c_1(\cald),$ the real first Chern class of the almost complex vector bundle $\cald.$ Of particular interest is the case $c_1(\cald)=0$. In this case the exact sequence (\ref{bascohseq2}) implies that there is an $a\in \bbr$ such that $c_1(\calf_\xi)=a[d\eta]_B.$ This gives rise to a rough classification of types of K-contact (Sasakian) structures according to whether $c_1(\calf_\xi)$ is positive definite, negative definite, zero, or indefinite. Note that indefinite Sasakian structures can only occur when the underlying contact bundle $\cald$ satisfies $c_1(\cald)\neq 0$. Here we are mainly concerned with the case $c_1(\cald)=0$. For some examples of indefinite Sasakian structures see Example 4.3 of \cite{BGO06}.

More generally, given two contact metric structures $\cals,\cals'\in \calc\calm(\calf_\xi),$ the following relations must hold
\begin{equation}\label{def1}
\xi'= a^{-1}\xi, \qquad \eta'=a\eta +\grz,
\end{equation}
where $a\in \bbr\setminus \{0\}$ and $\grz$ is a basic 1-form. Thus, $\calc\calm(\calf_\xi)$ decomposes as a disjoint union 
\begin{equation}\label{decomp1}
\calc\calm(\calf_\xi)= \bigcup_{a\in \bbr^+} \calc\calm(a^{-1}\xi) \sqcup \calc\calm(-a^{-1}\xi),
\end{equation}
where $\calc\calm(\xi)$ is the subset of $\calc\calm(\calf_\xi)$ consisting of those elements whose Reeb vector field is $\xi.$ Note that there is a natural involution $\calc\calm(\calf_\xi)\ra{1.5} \calc\calm(\calf_\xi)$ given by conjugation $\cals=(\xi,\eta,\Phi,g)\mapsto \cals^c=(-\xi,-\eta,-\Phi,g).$ It should be clear that conjugation restricts to an involution on the subspaces ${\mathcal S}(\calf_\xi)$ and $\calk(\calf_\xi).$ By fixing an orientation on $M$ and a co-orientation on $\cald$ we restrict ourselves to the subspace $\bigcup_{a\in \bbr^+} \calc\calm(a^{-1}\xi).$

Another important notion is that of {\it transverse homothety}. We say that a contact metric structure $\cals_a=(\xi_a,\eta_a,\Phi_a,g_a)$, labelled by $a\in \bbr^+$, is obtained from $\cals\in \calc\calm(\calf_\xi)$ by a {\it transverse homothety} if  
\begin{equation}\label{transhomoth}
\xi_a=a^{-1}\xi, \quad \eta_a=a\eta, \quad \Phi_a=\Phi, \quad
g_a=ag+(a^2-a)\eta\otimes \eta.
\end{equation}
Clearly, this gives rise to a 1-parameter family $\cals_a=(\xi_a,\eta_a,\Phi_a,g_a)$ of metrics in $\calc\calm(\calf_\xi).$
One can also view $\calc\calm(\pm\xi)$ as the set of transverse homothety classes in  
$\calc\calm(\calf_\xi)$. It is easy to see that a transverse homothety preserves the subspaces ${\mathcal S}(\calf_\xi)$ and $\calk(\calf_\xi).$ Note that the volume elements are related by $\mu_{g_a}=a^{n+1}\mu_g,$ so fixing a representative of a transverse homothety class fixes the volume and vice versa.

\subsection{Isotopy Classes of Contact Structures}
We are interested in varying the structures discussed above in several ways. First, if we fix a contact structure $\cald$ we can vary the almost complex structure on $\cald,$ namely $J=\Phi|_\cald.$ Then, if we fix the almost CR structure $(\cald,J)$, we can vary the 1-form $\eta$ within $\cald,$ and hence, the characteristic foliation $\calf_\xi.$ 

Alternatively, we can fix the characteristic foliation $\calf_\xi$ and vary the contact structure by deforming the 1-form $\eta\mapsto \eta_t=a\eta + t\grz,$ where $t\in [0,1],$ $a$ is a positive constant and $\grz$ is a basic 1-form. We also require that our variation is through contact forms by assuming that $\eta_t\wedge (d\eta_t)^n\neq 0$ for all $t\in [0,1].$ By Gray's Stability Theorem (cf. \cite{BG05}) all such contact structures are isomorphic on a compact manifold. So we are interested in the {\it isotopy class} $[\cald_t]_I$ of contact structures defined by $\{\cald_t=\ker~\eta_t\}.$ We let $\calc\calm_I$ denote the space of all contact metric structures within a fixed isotopy class of contact structures. On compact manifolds the space $\calc\calm_I$ is path connected, and different isotopy classes label different components of $\calc\calm.$ But as indicated above we are interested in the subspace $\calc\calm_{I}(\xi)$ consisting  of contact metric structures in $\calc\calm_I$ whose Reeb vector field is $\xi.$ Now more generally we have

\begin{definition}\label{cmib}
Fix a contact metric structure $\cals_0=(\xi_0,\eta_0,\Phi_0,g_0)\in \calc\calm_I.$ Then we define the space $\calc\calm_{I,B}$ to be the set of all contact metric structures $\cals=(\xi,\eta,\Phi,g)\in \calc\calm_I$ such that there is a smooth positive function $f$, and a 1-form $\grz_f$ with $\eta=f\eta_0+\grz_f$ and satisfying either of the following conditions
\begin{enumerate}
 \item $\grz_f$ is basic with respect to $\xi,$ or
\item $\grz_f$ has the form $\grz_f=f\grz$ where $\grz$ is a basic 1-form with respect to $\xi_0.$
\end{enumerate}
\end{definition}

Note that the two possibilities correspond to two distinct decompositions giving the non-commutative diagram:

\begin{equation}
\xymatrix{&(\cald,\calf_{\xi_f)}\ar[rd]^{+\grz} &\\
(\cald,\calf_{\xi})\ar[ru]^{\times f}\ar[rd]_{+\grz}&&(\cald_\grz,\calf_{\xi_f})\\
&(\cald_\grz,\calf_{\xi})\ar[ru]_{\times f}&}
\end{equation}

The `high road' (upper path) is $\eta\mapsto f\eta\mapsto f\eta +\grz_f,$ whereas, the `low road' is $\eta\mapsto \eta +\grz\mapsto f(\eta +\grz)$, and the two outcomes are different, in general. The former amounts to first deforming the foliation and then the CR structure, while the latter amounts to first deforming the CR structure and then the foliation. Note that the complex vector bundles $\cald$ and $\cald_\grz$ are isomorphic, so we have $c_1(\cald)=c_1(\cald_\grz).$ The basic first Chern class $c_1(\calf_\xi)$ depends only on the foliation, so it is invariant under the deformation $\eta\mapsto \eta+\grz.$ However, it is not understood generally how $c_1(\calf_\xi)$ changes under a deformation of the foliation through the Sasaki cone $\grk(\cald,J).$ Nevertheless, it is easy to see that $c_1(\calf_\xi)$ is invariant under a transverse homothety (\ref{transhomoth}), so it only depends on the projectivization of $\grk(\cald,J).$

We have 

\begin{definition}\label{isosas}
The space $\calc\calm_{I,B}$ is said to be of {\bf Sasaki type (K-contact type)} if some $\cals\in \calc\calm_{I,B}$ is Sasakian (K-contact), respectively. We denote by ${\mathcal S}_{I,B}, ({\mathcal K}_{I,B})$ the subspace of $\calc\calm_{I,B}$ consisting of Sasakian (K-contact) structures.
\end{definition}

Clearly, we have a natural inclusion ${\mathcal S}_{I,B}\subset {\mathcal K}_{I,B}.$

Accordingly, we define the subspace
\begin{equation}\label{sascr}
{\mathcal S}({\mathcal D},J)=\left\{\cals\in {\mathcal S}_{I,B} ~|~\text{$\cals$ is Sasakian with underlying CR structure $(\cald,J)$}
 \right\}.
\end{equation}

We have

\begin{proposition}\label{Isascon}
The space ${\mathcal S}_{I,B}$ is path connected. 
\end{proposition}

\begin{proof}
Fix a Sasakian structure $\cals_0=(\xi_0,\eta_0,\Phi_0,g_0)\in {\mathcal S}_{I,B}$ and let $\cals_1=(\xi_1,\eta_1,\Phi_1,g_1)$ be another Sasakian structure in ${\mathcal S}_{I,B}.$ Then $\eta_1=f\eta_0+\grz_f$ where $\grz_f$ satisfies one of the two conditions of Definition \ref{cmib}. If condition (1) holds, then as in the proof of Proposition 7.5.7 of \cite{BG05} the structure $\cals'=(f\eta_0,\xi_1,\Phi_1+\xi_1\otimes \grz_f\circ \Phi_1,g')$ belongs to a path component labelled by $H^1(M,\bbz)$, where the metric $g'$ is determined by Equation (\ref{conmetric}), is Sasakian. But since $\cals_1$ and $\cals_0$ are isotopic, $\cals'$ and $\cals_0$ lie in the same Sasaki cone which is path connected, the result follows. 

On the other hand if condition (2) of Definition \ref{cmib} holds, the contact metric structure $\cals''=(\xi_0,\eta_0+\grz,\Phi_0-\xi_0\otimes \grz\circ\Phi_0,g'')$, where $g''$ is determined by Equation (\ref{conmetric}), is Sasakian, since $\cals_0$ is Sasakian. But the structure $\cals_1$ is also Sasakian and $\eta_1=f(\eta_0+\grz).$ It follows that the Sasakian structures $\cals_1$ and $\cals''$ lie in the same Sasaki cone which proves the result.
\end{proof}

\subsection{Associated K\"ahler Geometries}
Given a contact manifold $M$ with a chosen contact form $\eta,$ the transverse geometry of the characteristic foliation $\calf_\xi$ is described by the symplectic vector bundle $(\cald,d\eta).$ Moreover,
on the cone $C(M)=M\times {\mathbb R}^{+}$ we also have a symplectic structure given by $d(r^2\eta)$ where $r\in \bbr^+.$ Now if $\cals=(\xi,\eta,\Phi,g)$ is a contact metric structure on $M,$ we have an almost K\"ahler structure on the cone $C(M)$ given by $(d(r^2\eta),I,g_{C(M)})$, where the cone metric $g_{C(M)}$ satisfies
\begin{equation}\label{conemetric}
g_{C(M)}=dr^2 + r^2 g \, , 
\end{equation}
and almost complex structure $I$ is defined by 
$$I(Y)=\Phi (Y) + \eta(Y) r\partial_r \, , \quad I(r\partial_r ) = -\xi \, .$$
Similarly, the transverse structure $(\cald,d\eta,J,g_\cald)$ is also almost K\"ahler. When $\cals$ is K-contact, the transverse structure is bundle-like. This is the situation about which we are most concerned. This leads to

\begin{definition}\label{qreg}
The characteristic foliation $\calf_\xi$ is said to be {\bf
quasi-regular} if there is a positive integer
$k$ such that each point has a foliated coordinate chart $(U,x)$
such that each leaf of $\calf_\xi$ passes through $U$ at most $k$
times. Otherwise $\calf_\xi$ is called {\bf irregular}. If $k=1$ then the foliation is called {\bf
regular}, and we use the terminology non-regular to mean quasi-regular, but 
not regular. We also say that the contact form $\eta$ or that the contact structure is quasi-regular (regular, irregular). 
\end{definition}

It is easy to see that a quasi-regular contact form $\eta$ is K-contact in the sense that there exists a K-contact metric that is compatible with $\eta.$ In this case the quotient space is well-defined, but we need to work in the category of orbifolds, and orbibundles. We refer to Chapter 4 of \cite{BG05} for details. Here we mention the following fundamental structure theorems
\cite{BG00a}:

\begin{theorem}\label{fundthm1}
Let $\cals=(\xi,\eta,\Phi,g)$ be a compact quasi-regular Sasakian (K-contact) structure on a manifold $M$ of dimension $2n+1$, and let $\calz$ denote the space of leaves of the characteristic foliation. Then
the leaf space $\calz$ is a Hodge (symplectic) orbifold with
K\"ahler (almost K\"ahler) metric $h$ and K\"ahler form $\gro$ which defines an integral class
$[\gro]$ in $H^2_{orb}(\calz,\bbz)$ so that $\pi:(M,g) \ra{1.3}
(\calz,h)$ is an orbifold Riemannian submersion. The fibers of $\pi$ are
totally geodesic submanifolds of $M$ diffeomorphic to $S^1.$
\end{theorem}

Note that when $\cals$ is Sasakian it follows from the Kodaira-Baily Embedding Theorem \cite{Bai57,BG05} that $\calz$ is a projective algebraic variety with an orbifold structure. Conversely,  the inversion theorem holds:

\begin{theorem}\label{fundthm2} 
Let $(\calz,h)$ be a compact Hodge (symplectic) orbifold.
Let $\pi:M\ra{1.3} \calz$ be the $S^1$ orbibundle whose first Chern
class is $[\gro],$ and let $\eta$ be a connection 1-form in $M$ whose
curvature is $2\pi^*\gro,$ then $M$ with the metric
$\pi^*h+\eta\otimes\eta$ is a Sasakian (K-contact) orbifold.  Furthermore, if all the local uniformizing groups inject into the group of the bundle $S^1,$ the total space $M$ is a smooth Sasakian (K-contact) manifold.
\end{theorem}

\begin{remark} The structure theorems and the discussion above show that
Sasakian geometry is ``sandwiched'' between two K\"ahler geometries as illustrated by the following diagram

\begin{equation}\label{sas}
\begin{matrix}
C(M)&\hookleftarrow&M \\
&{}&\decdnar\pi \\
{}&{}&\calz \\
\end{matrix}
\end{equation}
Kris liked to refer to this as the ``K\"ahler-Sasaki sandwich''.
\end{remark}

\subsection{Extremal Sasakian Metrics}
In recent work \cite{BGS06,BGS07b} Kris and I in collaboration with Santiago Simanca have developed a theory of extremal Sasakian structures. Here the theory parallels to a certain extent the K\"ahlerian case introduced by Calabi \cite{Cal82}. We refer to \cite{Tia00,Sim04} and references therein for a presentation of this case. 

We are interested in the subspace ${\mathcal S}(\xi,\bar{J})$ of ${\mathcal S}_{I,B}$ with a fixed Reeb vector field and a fixed complex structure $\bar{J}$ on the normal bundle $\nu(\calf_\xi).$ Following \cite{BGS06} we assume that $M$ is compact of Sasaki type and denote by ${\mathfrak M}(\xi,\bar{J})$ the set of all compatible Sasakian metrics arising from structures in ${\mathcal S}(\xi, \bar{J})$, and 
define the ``energy functional'' $E:{\mathfrak M}(\xi,\bar{J})\ra{1.4} \bbr$ by
\begin{equation}\label{var}
E(g) ={\displaystyle \int _M s_g ^2 d{\mu}_g ,}\, 
\end{equation}
i.e. the square of the $L^2$-norm of the scalar curvature $s_g$ of $g$. Critical points $g$ of this functional are called {\it extremal Sasakian metrics}, and the associated Sasakian structure $\cals=(\xi,\eta,\Phi,g)$ is called an {\it extremal Sasakian structure}.  Similar to the K\"ahlerian case the Euler-Lagrange equations for this function give
\begin{theorem}\label{ELeqn}
A Sasakian metric $g\in \calm(\xi,\bar{J})$ is a critical point for the energy functional (\ref{var}) if and only if the gradient vector field $\partial^\#_gs_g$ is transversely holomorphic. In particular, Sasakian metrics with constant scalar curvature are extremal.
\end{theorem}

If $\cals$ is an extremal Sasakian metric then so is the Sasakian structures obtain by a transverse homothety. This follows from the easily verified relation $\partial^\#_{g_a}s_{g_a}= a^{-2}\partial^\#_gs_g.$ Notice also that since there are no transversely holomorphic vector fields for negative or null Sasakian structures, the only extremal negative or null Sasakian metrics are ones of constant scalar curvature. 

Also analogously to the K\"ahler case, for any transversally holomorphic vector
field $\d^{\#}_g f$ with Killing potential $f$, we have (see \cite{BGS06}) a Sasaki-Futaki 
invariant given by
\begin{equation}\label{SFinv}
\gS{\mathfrak F}_{\xi}(\d_g^{\#}f)=- \int_M f(s_g-s_{0})d\mu_g \, , 
\end{equation}
where $g$ is any Sasakian metric in ${\mathfrak M}(\xi,\bar{J})$ and $s_0$ is the projection of the scalar curvature $s_g$ onto the constants.
This expression uniquely defines a character on the Lie algebra $\gh(\xi,\bar{J})$ of transversally holomorphic vector fields. Moreover, we have 

\begin{theorem}\label{SFinvthm}
If $\cals=(\xi, \eta, \Phi, g)$
is an extremal Sasakian structure in ${\mathcal S}(\xi, \bar{J})$, 
the scalar curvature $s_g$ is constant if and only
if $\gS{\mathfrak F}_{\xi}$ vanishes identically.
\end{theorem}

Of particular interest to us is the case that the contact bundle $\cald$ has vanishing first Chern class $c_1(\cald),$ or a bit more generally $c_1(\cald)$ is a torsion class. A special case of Sasakian structures $\cals=(\xi,\eta,\Phi,g)$ where $c_1(\cald)$ is a torsion class and $g$ has constant scalar curvature is the case of so-called $\eta$-Einstein metrics. These are all extremal Sasakian metrics that satisfy
\begin{equation}\label{etaEin}
\Ric_g=\lambda g+\nu\eta\otimes \eta
\end{equation}
for some constants $\lambda$ and $\nu$, where $\Ric_g$ denotes the Ricci curvature tensor. The scalar curvature 
$s_g$ of these metrics is constant and given by
$s_g=2n(1+\lambda)$.
We refer the reader to \cite{BGM06}, and references therein, for further 
discussion of these metrics.  

Given a Sasakian structure $\cals=(\xi,\eta,\Phi,g)$ in ${\mathcal S}(\xi,\bar{J})$, we can vary the 1-form $\eta$ within the CR structure $(\cald,J),$ and hence, its Reeb vector field. If we vary through Sasakian structures, then the Reeb vector fields vary through the Sasaki cone $\grk(\cald,J)$. One of the main results of \cite{BGS06} says that the set of extremal Sasakian structures is open in the Sasaki cone. 

In \cite{BGS07b} another variational principal was introduced which gives critical points keeping the CR structure fixed and varying the characteristic foliation in the Sasaki cone. This is similar to the K\"ahler case described in \cite{Sim96,Sim00}. Here is the setup. Fix a Sasakian structure $\cals\in {\mathcal S}(\cald,J),$ and let $T_k$ be a maximal torus in $\gC\gR(\cald,J).$ There is a $T_k$ equivariant moment map $\mu:M\ra{1.5} \gt_k^*\approx \bbr^k$ defined by $\langle\mu, \grt\rangle=\eta(X^\grt),$ where $X^\grt$ is the vector field on $M$ associated to $\grt\in \gt_k.$ Now $\mu$ defines the vector subspace $\calh_\cals(M)\subset C^\infty(M)\subset L^2_g(M)$ of {\it Killing potentials} spanned by the functions $\eta(X^\grt).$ Note that the vector space $\calh_\cals(M)$ depends only on the transverse homothety class; however, for ease of notation we just use the subscript $\cals.$ Since $\xi\in \gt_k$, the space $\calh_\cals(M)$ contains the constants. Moreover, we have a splitting with respect to the $L^2_g$-norm, 
$$C^\infty(M) = \calh_\cals(M)\oplus \calh_\cals(M)^\perp,$$
and we have a natural projection $\pi_g:C^\infty(M) \ra{1.5} \calh_\cals(M).$ It is also easy to check that this splitting depends only on the transverse homothety class. Note that when $k=1$ the space of Killing potentials $\calh_\cals(M)$ consists only of the constant functions. We note that the energy functional $E$ of Equation (\ref{var}) has a lower bound, viz.
\begin{equation}\label{extlb}
 E(g)=\int_M s_g^2 d\mu_g \geq
\int_M (\pi_{g}s_g)^2 d\mu_g \,.
\end{equation}
Then the metric $g$ in ${\mathfrak M}(\xi,\bar{J})$ is extremal if
and only if the scalar curvature $s_g$ lies in $\calh_\cals(M),$ that is $s_g$ is a linear combination of Killing potentials.

When the lower bound in Equation (\ref{extlb}) is reached we can ask about varying the Reeb vector field within the Sasaki cone. This is done in \cite{BGS07b}, and is the adabtation to Sasakian geometry of an idea of Simanca \cite{Sim96,Sim00} in the K\"ahlerian case.

We define the energy $\cale$ of a Reeb vector field in the Sasakian cone to be the
functional given by this optimal lower bound, namely $\cale:\kappa({\mathcal D},J) \ra{1.5} \bbr$ defined by
\begin{equation}\label{en}
\cale(\xi) ={\displaystyle \int_M (\pi_{g}s_g)^2 d\mu_g }\,
\end{equation}
The Reeb vector field $\xi \in \kappa ({\mathcal D},J)$ of a Sasakian structure $\cals=(\xi,\eta,\Phi,g)$ is said to be
{\it strongly extremal} if it is a critical point of the functional {\rm (\ref{en})} over the space of Sasakian structures in $\kappa ({\mathcal D},J)$ that fix the volume of $M$. A Sasakian structure $\cals$ is {\it strongly extremal} if its metric is extremal and its Reeb vector field $\xi$ is strongly extremal. Notice that when $\cals$ is an extremal Sasakian structure, the two integrals (\ref{var}) and (\ref{en}) are the same, but the variations are different. The Euler-Lagrange equation for the functional (\ref{en}) is
\begin{equation}
\pi_g \left[ -2n \Delta_B \pi_g s_g -(n-1)(\pi_g s_g)^2\right] +
4n \pi_g s_g =\lambda \, ,
\label{el}
\end{equation}
where $\grl$ is a constant. Thus, $\pi_gs_g$ a constant gives a solution to the Euler-Lagrange equation. This occurs in two cases: (1) The Sasaki cone is one dimensional, and (2) the scalar curvature $s_g$ is itself constant. The first case is rather trivial since the only variation is a transverse homothety, and all these satisfy the equations. Moreover, they do not generally give strongly extremal Sasakian structures, as we shall show later. It is the second case that is of interest. If $s_g$ is constant then the Sasakian structure $\cals$ is strongly extremal. I know no examples of strongly extremal Sasakian structures with $\pi_gs_g$ non-constant. Important known examples of strongly extremal Sasakian structures are the Sasakian-$\eta$-Einstein structures given by Equation (\ref{etaEin}), but they are not the only ones. As shown in \cite{BGS07b} the Wang-Ziller manifolds \cite{WaZi90,BG05} give examples of strongly extremal Sasakian structures with $c_1(\cald)\neq 0.$

\section{Links of Isolated Singularities}

Let $(M,\cals)$ be a compact Sasakian manifold and consider the cone $C(M)=M\times \bbr^+$ over $M.$ Let $Y=C(M)\cup\{r=0\}$ be $C(M)$ with  the cone point added. $Y$ is a complex analytic space with an isolated singularity at $\{0\},$ or equivalently a complex space germ $(Y,0)$. If the Sasakian structure is quasi-regular, then $Y$ has a natural $\bbc^*$-action with a fixed point at $r=0.$ In this case $\Psi-i\xi$ as a holomorphic vector field on $Y$, where $\Psi=r\partial_r,$ is an infinitesimal generator of the $\bbc^*$-action, and $C(M)$ is the total space of a principal $\bbc^*$-orbibundle over the projective algebraic orbifold $\calz.$ Alternatively, this can be phrased in terms of $\bbc^*$-Seifert bundles, cf. Section 4.7 of \cite{BG05}. This general Seifert bundle construction is given in an unpublished work of Koll\'ar \cite{Kol04c}.

Though not entirely understood, it is well-known that much of the topology of $M$ is controlled by the nature of the conical singularity of $Y$ at $r=0$, cf. \cite{GrSt83}. First I mention a recent result of Marinescu and Yeganefar \cite{MaYe07} which says that for any Sasakian manifold there is a CR embedding into $\bbc^N$ for some $N.$ In particular, the topology of $M$ is related to the minimal embedding dimension $N.$ Moreover, the result in \cite{MaYe07} implies that any contact structure of Sasaki type is holomorphically fillable. Explicitly, we have

\begin{definition}\label{fillable}
 A contact manifold $(M,\cald)$ is said to be {\bf holomorphically fillable} if it is contactomorphic to the boundary of a compact strictly pseudoconvex complex manifold $V$. If $V$ can be taken to be Stein then $(M,\cald)$ is said to be {\bf Stein fillable}. 
\end{definition}

Then as noted in \cite{MaYe07} a well-known result of Harvey and Lawson \cite{HaLa75} together with the Marinescu-Yeganefar theorem gives

\begin{theorem}\label{MYthm}
Every compact Sasakian manifold is holomorphically fillable.
\end{theorem}

In a similar vein a recent result of Niederkr\"uger and Pasquotto \cite{NiPa07} shows that a compact K-contact manifold is symplectically fillable.
Generally, it appears that Sasakian manifolds are not necessarily Stein fillable as recently shown by Popescu-Pampu \cite{P-P07c}. Thus, they are not, in general, smoothable in the sense of \cite{GrSt83}. 

Further study of these important contact invariants is currently in progress. The one case that has been studied in detail is that of isolated hypersurface singularities arising from weighted homogeneous polynomials which I now discuss.

\subsection{Weighted Hypersurface Singularities}
In \cite{BG01b} Kris and I described a method for proving the existence of Sasaki-Einstein metrics on links of isolated hypersurface singularities that arise from weighted homogeneous polynomials. Many of the results that we obtained since then have involved the natural Sasakian geometry occurring on such links. Here I provide a brief review of the relevant geometry referring to Chapters 4 and 9 of \cite{BG05} for a more thorough treatment. Recall  
that a polynomial $f\in {\mathbb C}[z_0,\ldots,z_n]$ is said to be a 
{\it weighted homogeneous polynomial} of {\it degree} $d$ and 
{\it weight} ${\bf w}=(w_0,\ldots,w_n)$ if for any $\lambda \in 
{\mathbb C}^*=\bbc\setminus \{0\},$ we have
$$
f(\lambda^{w_0} z_0,\ldots,\lambda^{w_n}
z_n)=\lambda^df(z_0,\ldots,z_n)\, .
$$
We are interested in those weighted homogeneous polynomials $f$ whose 
zero locus in ${\mathbb C}^{n+1}$ has only an isolated singularity at the 
origin. We define the {\it link} $L_f({\bf w},d)$ as 
$f^{-1}(0)\cap S^{2n+1}$, where $S^{2n+1}$ is the 
$(2n+1)$-sphere in $\bbc^{n+1}$. By the Milnor fibration theorem 
\cite{Mil68}, $L_f({\bf w},d)$ is a closed $(n-2)$-connected manifold that 
bounds a parallelizable manifold with the homotopy type of a bouquet of $n$-spheres. Furthermore, $L_f({\bf w},d)$ admits a 
Sasaki-Seifert structure 
${\oldmathcal S}=(\xi_{\bf w},\eta_{\bf w},\Phi_{\bf w},g_{\bf w})$ in a 
natural way 
\cite{Abe77,Tak,Var80} (I only recently became aware of the last reference). This structure is quasi-regular, and the contact 
bundle $({\mathcal D},J)$ satisfies $c_1({\mathcal D})=0$. As mentioned previously this latter property implies that $c_1({\mathcal F}_{\xi_{\bf w}})=a[d\eta_{\bf w}]_B$ for some constant $a$, where ${\mathcal F}_{\xi_{\bf w}}$ is the 
characteristic foliation. The sign of $a$ determines the negative, null, 
and positive cases. For any list of positive integers $(l_1,\ldots,l_k)$ it is convenient to set $|\bfl|=\sum_il_i$. Then more generally, for the links of isolated complete interection singularities defined by weighted homogeneous polynomials $\bff=(f_1,\ldots,f_k)$ we have

\begin{proposition}\label{cis}
Let $L_\bff$ be the link of an isolated complete interection singularity defined by weighted homogeneous polynomials $\bff=(f_1,\ldots,f_k)$ of multidegree $\bfd=(d_1,\ldots,d_k)$. Then $L_\bff$ has a natural Sasakian structure $\cals$ which is either positive, negative, or null. In particular, $c_1(\cald)=0.$ Moreover, $\cals$ is positive, negative, or null, depending on whether $|\bfw|-|\bfd|$ is positive, negative, or null, respectively.
\end{proposition}

\begin{remark}\label{Goren}
The condition $c_1(\cald)=0$ is equivalent to the condition $c_1(TC(M))=0.$ Since for any Sasakian manifold $M$ the cone $C(M)$ is K\"ahler, $c_1(\cald)=0$ implies the existence of a nowhere vanishing holomorphic $(n+1)$-form $\grO,$ as well as the fact that the cone singularity is what is called a {\it Gorenstein singularity}. We discuss this further in Section \ref{torsas} in the context of toric geometry.  
\end{remark}

The only transversally holomorphic vector fields of a negative or a
null link of an isolated hypersurface singularity are those
generated by the Reeb vector field. These links have a one-dimensional
Sasaki cone, and by the transverse Aubin-Yau theorem 
\cite{BGM06}, any point in this cone can be represented by a Sasakian
structure whose metric is $\eta$-Einstein (see examples of such links 
in \cite{BoGa05a,BG05}). So the entire Sasaki cone consists of strongly extremal Sasakian structures \cite{BGS07b}. Negative Sasaki-$\eta$-Einstein metrics give rise to Lorenzian Sasaki-Einstein metrics \cite{Bau00,BGM06}, and in dimension 5 R. Gomez \cite{Gom08} has recently shown that for all $k\in \bbn$ such metrics exist on the $k$-fold connected sum $k(S^2\times S^3).$ Moreover, using work of Koll\'ar \cite{Kol05b}, J. Cuadros \cite{Cua08} has recently shown that the only simply connected 5-manifolds that can admit null Sasakian structures are $k(S^2\times S^3)$ when $3\leq k\leq 21,$ and all of these do admit such structures. Moreover, both the null and negative cases can occur with moduli. 

\subsection{On the Existence of Sasaki-Einstein Metrics} In the case of positive links there are well-known obstructions to the existence of Sasaki-$\eta$-Einstein metrics (cf. \cite{BG05}). One such obstruction has already been discussed, namely, the Sasaki-Futaki invariant (\ref{SFinv}). Another is the Matsushima-Lichnerowicz type invariant \cite{BGS06}. Both of these invariants rely on the existence of non-trivial transverse holomorphic vector fields. More recently it was observed in \cite{GMSY06} that classical estimates of Bishop and Lichnerowicz, which do not depend on the existence of symmetries, may also obstruct the existence of Sasaki-Einstein metrics. For us the Lichnerowicz obstruction seems more convenient. In particular, 

\begin{proposition}\label{Lich.Ob.Link.Ineq} Let $L_f({\bf w},d)$ be a 
link of an isolated hypersurface singularity with its natural Sasakian 
structure, and let $I=|{\bf w}|-d$ be its index. If 
$$I> n \, \min_{i}{\{w_i\}}\, ,$$
then $L_f({\bf w};d)$ cannot admit any Sasaki-Einstein metric. 
\end{proposition}

So using Proposition \ref{Lich.Ob.Link.Ineq}, we can obtain examples of positive links that admit no Sasaki-Einstein structure \cite{BGS07b}. When the Sasaki cone is one dimensional this gives isotopy classes of contact metric structures of Sasaki type that admit no compatible Sasaki-Einstein structure, hence, no extremal Sasakian structure. When the Sasaki cone has dimension greater than one, there remains the possibility of deforming through the Sasaki cone to obtain an (extremal) Sasaki-Einstein structure. Precisely this was accomplished in the toric case in \cite{FOW06,CFO07} under the assumption that $c_1(\cald)=0.$

We now want to review some known methods for proving the existence of Sasaki-Einstein metrics. I refer to Chapter 5 of \cite{BG05} and references therein for details and more discussion. We restrict ourselves to positive quasi-regular Sasakian structures on compact manifolds in which case Theorems \ref{fundthm1} and \ref{fundthm2} apply.
If the K\"ahler class of the projective 
algebraic orbifold $\calz$ is a primitive integral class in the second 
orbifold cohomology group, 
the sufficient conditions, which gives a measure of the singularity of the pair 
$(\calz,K^{-1}+\Delta)$, where $K^{-1}$ is an anti-canonical divisor, and 
$\Delta$ is a branch divisor, are known as 
{\it Kawamata log terminal} (klt) conditions. Rather than give a general discussion I will concentrate on some easily applicable special cases. First we can obtain a rather crude estimate for any link $L_f(\bfw,d)$ of an isolated hypersurface singularity from a weighted homogeneous polynomial $f.$ 

\begin{proposition}\label{whsklt}
Let $L_f({\bf w},d)$ be a 
link of an isolated hypersurface singularity with its natural Sasakian 
structure, and let $I=|{\bf w}|-d$ be its index. Then $L_f({\bf w},d)$ admits a Sasaki-Einstein metric if 
$$Id < \frac{n}{n-1}{\rm min}_{i,j}(w_iw_j).$$ 
\end{proposition}

This estimate is far from optimal, and in some special cases we can do much better.
One such case is that of the well-known {\it Brieskorn-Pham} (BP) polynomial defined by a weighted homogeneous polynomial of the form 
\begin{equation}
f=z_0^{a_0}+\cdots +z_n^{a_n}\, , \quad a_i \geq 2\, ,
\label{bp}
\end{equation}
In this case the 
exponents $a_i$, the weights $w_i$, and degree are related by $d=a_iw_i$ 
for each $i=0,\ldots,n.$ It is convenient to change notation slightly 
and denote such links by $L_f({\bf a})$, where
${\bf a}=(a_0, \ldots , a_n)$. These are special but quite important examples of links. Their
klt conditions were described in \cite{BGK05}. The 
base of a BP link 
$L_f({\bf a})$ admits a positive K\"ahler-Einstein orbifold metric if
\begin{equation}\label{bgkest}
 1<\sum_{i=0}^n\frac{1}{a_i}<
1+\frac{n}{n-1}\min_{i,j}\Bigl\{\frac{1}{a_i},
\frac{1}{b_ib_j}\Bigr\}\,.
\end{equation}
where $b_j=\gcd(a_j,C^j)$ and $C^j={\rm lcm}\{ a_i : \, i\neq j\}$.
This condition leads to a rather large 
number of examples of Sasaki-Einstein metrics on homotopy spheres
\cite{BGK05,BGKT05} and rational homology spheres \cite{BG06b,Kol05b}.

In the special case when the integers $(a_0,\ldots,a_n)$ are pairwise 
relatively prime, Ghigi and Koll\'ar \cite{GhKo05} obtained a sharp 
estimate. In this case, the BP link is always a homotopy sphere, and if
we combine this with Proposition \ref{Lich.Ob.Link.Ineq}
above, we see that, when the $a_i$s are pairwise relatively prime, a BP
link $L_f({\bf a})$ admits a Sasaki-Einstein metric if and only if 
$$1<\sum_{i=0}^n\frac{1}{a_i}<
1+n~\min_{i}\Bigl\{\frac{1}{a_i}\Bigr\}\, .$$
Other applications of the klt estimate (\ref{bgkest}) can be found in 
\cite{BoGa05a,BG05}. 

\subsection{The Topology of Links: Orlik's Conjecture}
By using the Milnor fibration theorem, Milnor and Orlik \cite{MiOr70} obtained more refined topological information about links of isolated hypersurface singularities arising from weighted homogenous polynomials though a study of the Alexander polynomial of the link. They computed
the Betti numbers of the manifold underlying the link, and in the 
case of rational homology spheres, gave 
the order of the relevant homology groups. 
A bit later, Orlik \cite{Or72} postulated a combinatorial 
conjecture for computing the torsion, an algorithm that we describe next
(see Chapter 9 of \cite{BG05} for more detail). 

Given a link $L_f({\bf w}, d)$, we define its {\it fractional weights}
to be 
\begin{equation}\label{frct.1}
\Bigl(\frac{d}{w_0},\cdots, \frac{d}{w_n}\Bigr)\equiv
\Bigl(\frac{u_0}{v_0},\cdots, \frac{u_n}{v_n}\Bigr),
\end{equation}
where
\begin{equation}\label{uv.wd.conv}
u_i=\frac{d}{{\rm gcd}(d,w_i)},\qquad v_i=\frac{w_i}{{\rm
gcd}(d,w_i)}.
\end{equation}
We denote by $({\bf u}, {\bf v})$ the tuple $(u_0,\ldots,u_n,
v_0,\ldots,v_n)$. By \eqref{uv.wd.conv}, we may 
go between $({\bf w},d)$ and $({\bf u}, {\bf v})$. 
We will sometimes write $L_f({\bf u}, {\bf v})$ for 
$L_f({\bf w}, d)$. Then the $(n-1)^{st}$ Betti number is
\begin{equation}\label{Betti}
b_{n-1}(L_f)= \sum (-1)^{n+1-s}\frac{u_{i_1}\cdots
u_{i_s}}{v_{i_1}\cdots v_{i_s}{\rm lcm}(u_{i_1},\ldots,u_{i_s})}\,
\end{equation}
where the sum is taken over all the $2^{n+1}$ subsets
$\{i_1,\ldots,i_s\}$ of $\{0,\ldots,n\}.$
In \cite{Or72} Orlik conjectured an algorithmic procedure for computing the torsion of these links. 
   
\begin{conjecture}[Orlik]\label{orl.alg}
Let $\{i_1,\ldots,i_s\} \subset\{0,1,\ldots,n\}$ be an ordered set of 
$s$ indices, that is to say, $i_1<i_2<\cdots<i_s$. Let us denote by $I$ its 
power set {\rm (}consisting of all of the $2^s$ subsets of the set{\rm )}, 
and by $J$ the set of all proper subsets. 
Given a $(2n+2)$-tuple $({\bf u}, {\bf v})=(u_0,\ldots,u_n,
v_0,\ldots,v_n)$ of integers, we define inductively a set of $2^s$ positive 
integers, one for each ordered element of $I$, as follows: 
$$c_{\emptyset}=\gcd{( u_0,\ldots, u_n )}\, ,$$
and if $\{i_1,\ldots,i_s\}\in I$ is ordered, 
then
\begin{equation}\label{Orlikc's}
c_{i_1,\ldots,i_s}=\frac{\gcd{(u_0,\ldots,\hat{u}_{i_1},\ldots,\hat{u}_{i_s},
\ldots,u_n)}}{\prod_J c_{j_1,\ldots j_t}}\, .
\end{equation}
Similarly, we also define a set of $2^s$ real numbers by
$$k_\emptyset=\epsilon_{n+1}\, ,$$
and
\begin{equation}
k_{i_1,\ldots,i_s}= \epsilon_{n-s+1}\sum_{J}(-1)^{s-t}
\frac{u_{j_1}\cdots u_{j_t}}{v_{j_1}\cdots v_{j_t}{\rm
lcm}( u_{j_1},\ldots,u_{j_t})}\, ,
\end{equation}
where
$$\epsilon_{n-s+1}=\left\{ \begin{array}{cl}
0 & \text{if $n-s+1$ is even,} \\
1 & \text{if $n-s+1$ is odd,}
\end{array}
\right.
$$
respectively. 
Finally, for any $j$ such that $1\leq j\leq r=\lfloor \max{\{k_{i_1,\ldots,i_s}\}}\rfloor$,
where $\lfloor x\rfloor$ is the greatest integer less than or equal to $x$, we set
\begin{equation}\label{torsionorders}
d_j=\prod_{ k_{i_1,\ldots,i_s}\geq j}c_{i_1,\ldots,i_s}\, .
\end{equation}
Then 
$$H_{n-1}(L_f({\bf u},{\bf v}),{\mathbb Z})_{\rm tor}=
{\mathbb Z}/d_1 \oplus \cdots \oplus {\mathbb Z}/d_r \, .$$
\end{conjecture}

This conjecture is known to hold in certain special cases 
\cite{Ran75,OrRa77a} as well as in dimension 3 \cite{Or70}. Before his tragic accident, Kris proved that Conjecture 
\ref{orl.alg} holds in dimension 5 also by using a recent theorem of Koll\'ar \cite{Kol05b}. The argument is given in the appendix of \cite{BGS07b}.

\begin{proposition}\label{Orlikconjknown}
{\rm Conjecture \ref{orl.alg}} holds in the following  cases:
\begin{enumerate}
\item In dimension $3$, that is for $n=2.$ 
\item In dimension $5$, that is for $n=3.$ 
\item For $f({\bf z})=z_0^{a_0}+\cdots +z_n^{a_n}$, BP polynomials.
\item For $f({\bf z})=z_0^{a_0}+z_0z_1^{a_1}+z_1z_2^{a_2}+\cdots
+z_{n-1}z_n^{a_n}$, OR polynomials. 
\end{enumerate}
\end{proposition}

\noindent It is also known to hold for certain complete intersections given by 
generalized Brieskorn polynomials \cite{Ran75}. 

Evan Thomas wrote a computer program for Kris and I that computes $b_{n-1}$ according to Equation \ref{Betti} and the $d_i$s according to Equation \ref{torsionorders}. We have made use of this program previously in \cite{BoGa05a,BG05}, as well as in the present work for the cases where Orlik's Conjecture is known to hold.

\section{Sasakian Geometry in Low Dimensions}
In this section we mainly discuss Sasakian geometry in the three lowest dimensions, 3,5, and 7. We have a fairly complete understanding of Sasakian geometry in dimension 3 thanks to the work of Geiges \cite{Gei97} and Belgun \cite{Bel01}, while the work of Koll\'ar \cite{Kol04,Kol05b} as well as Kris, myself, and collaborators \cite{BG01b,BGN03a,BGN03c,BGN02b,BoGa05a,BG06b} have show that Sasakian geometry in dimension 5 is quite tractable. With the exception of spheres \cite{BGK05,BGKT05} much less is known for dimension 7 and higher. Nevertheless, dimension 7 is of interest for several reasons, not the least of which is its connection to physics.

\subsection{Sasakian Geometry in Dimension 3}
The topology of compact 3-manifolds has been at the forefront of mathematical research for some time now. Even after Perelman's proof of the Poincar\'e conjecture and Thurston's geometrization conjecture, there still remains much to be done. Much recent work has focused on contact topology. The reader is referred to the books \cite{ABKLR94,OzSt04} as well as the more recent \cite{Gei08} for a more thorough treatment of this important topic. I am far from an expert in this area, and here my intent is simply to provide a brief review as it relates to 3-dimensional Sasakian geometry. A well-known theorem of Martinet says that every compact orientable 3-manifold admits a contact structure, and Eliashberg \cite{Eli89,Eli92} set out to classify them. In particular, he divided contact structures on 3-manifolds into two types, {\it overtwisted} and {\it tight} contact structures. He classified the overtwisted structures in \cite{Eli89} and in \cite{Eli92} proved that there is precisely one (the standard one) tight contact on $S^3$ up to isotopy. Moreover, Gromov \cite{Gro85}, with more details provided by Eliashberg \cite{Eli90}, proved that a symplectically fillable\footnote{Here by symplectically fillable we mean weak symplectically fillable. It is known that there are two distinct types of symplectic fillability, weak and strong with a proper inclusion $\{\text{strongly symplectically fillable}\}\subsetneq \{\text{weakly symplectically fillable}\}$; however, the difference is not important to us at this stage.} contact structure is tight, and that a holomorphically fillable contact structure is symplectically fillable. In dimension 3 holomorphic fillability and Stein fillability coincide. Furthermore, by Theorem \ref{MYthm} a contact structure of Sasaki type on a 3-manifold is Stein fillable; thus, we have the following nested types of tight contact structures:

$$
\begin{array}{ccccccc}
& & & & & &\\
\text{\{Sasaki type\}} &\subsetneq &\text{\{Stein fillable\} } & \subsetneq & \text{\{Symplectically fillable}\} & \subsetneq & \text{Tight }\\
& & & & & &\\
\end{array}
$$

Generally, as indicated all of the above inclusions are proper. One question of much interest concerns the different types of contact structures that can occur on a given 3-manifold. It appears that overtwisted contact structures are ubiquitous, and they are never of Sasaki type, so we shall restrict our discussion to tight contact structures.

In constrast to arbitrary contact structures, Geiges \cite{Gei97} showed that Sasakian structures in dimension 3 can only occur on 3 of the 8 model geometries of Thurston, namely 
\begin{enumerate}
\item $S^3/\grG$ with $\grG\subset \gI\gs\go\gm_0(S^3)=SO(4).$ \item
$\widetilde{SL}(2,\bbr)/\grG,$ where $\widetilde{SL}(2,\bbr)$ is the universal cover of
$SL(2,\bbr)$ and $\grG\subset \gI\gs\go\gm_0(\widetilde{SL}(2,\bbr)).$
\item ${\rm Nil}^3/\grG$ with $\grG\subset \gI\gs\go\gm_0({\rm Nil}^3).$
\end{enumerate} 
Here $\grG$ is a discrete subgroup of the connected component
$\gI\gs\go\gm_0$ of the corresponding isometry group with respect to a
`natural metric', and ${\rm Nil}^3$ denotes the $3$ by $3$
nilpotent real matrices, otherwise known as the Heisenberg group. Then Belgun \cite{Bel01} gave a type of uniformization result. Here we give this in the form presented in \cite{BG05} in terms of the so-called $\Phi$-sectional curvature $K(X,\Phi X)$:

\begin{theorem}\label{Belgunthm}
Let $M$ be a $3$-dimensional oriented compact manifold admitting a Sasakian
structure $\cals=(\xi,\eta,\Phi,g).$ Then
\begin{enumerate}
\item If $\cals$ is positive, $M$ is spherical, and there is a
Sasakian metric of constant $\Phi$-sectional curvature $1$ in the
same isotopy class as $\cals.$ 
\item If $\cals$ is negative, $M$
is of $\widetilde{SL_2}$ type, and there is a Sasakian metric of
constant $\Phi$-sectional curvature $-4$ in the same isotopy
class as $\cals.$ 
\item If $\cals$ is null, $M$ is a nilmanifold, and there is a
Sasakian metric of constant $\Phi$-sectional curvature $-3$ in the
same isotopy class as $\cals.$
\end{enumerate}
\end{theorem}

This theorem has an important corollary which is certainly not true in higher dimension.

\begin{corollary}
Let $M$ be a oriented compact 3-manifold. Then, up to isomorphism, $M$ admits at most one isotopy class of contact structures of Sasaki type.
\end{corollary}

\begin{proof}
The three types given in Theorem \ref{Belgunthm} are actually topologically distinct. So if $\cals$ and $\cals'$ are two Sasakian structures on $M$, they must be of the same type. Thus, by Theorem \ref{Belgunthm} the Sasakian structures $\cals$ and $\cals'$ are both of the same type and isotopic to one of constant $\Phi$-sectional curvature either $1,-3$ or $-4$. But by a theorem of Tanno \cite{Tan69b}, such structures are unique up to isomorphism.
\end{proof}

In constrast to this result there can be many isotopy classes of tight contact structures on 3-manifolds \cite{CGH08}, depending on the manifold. However, there are cases where there is only one, most notably the following theorem of Eliashberg \cite{Eli92}: 

\begin{theorem}[Eliashberg]\label{Elithm} Every tight contact structure on $S^3$ is isotopic to the standard contact structure, which is of Sasaki type. 
\end{theorem}

Subsequently, there was much further development along these lines by Giroux \cite{Gir00,Gir01} and Honda \cite{Hon00a,Hon00b}. What is most relevant to us is that they determined the number of isotopy classes of tight contact structures on the total space of circle bundles over Riemann surfaces that are transverse to the fibres. By Theorem \ref{fundthm2} we know that when the bundle is non-trivial at least one of these tight contact structures is of Sasaki type. However, Eliashberg \cite{Eli92} noticed the following Bennequin-type inequality: Let $\grS\hookrightarrow M$ be an embedded compact Riemann surface of genus $g$ in a tight contact manifold $M.$ Then 
\begin{equation}\label{Bennineq}
|c_1(\cald)([\grS])|=\begin{cases} 
                    0 &\text{if $\grS=S^2$;} \\
                   \leq 2g-2 &\text{otherwise}.
                   \end{cases}
\end{equation}
This implies that there is only a finite number of classes in $H^2(M,\bbz)$ that can represent $c_1(\cald)$ for tight contact structures on $M.$ But Sasakian structures must have $c_1(\cald)=0$, since $H^2(\grS,\bbz)=\bbz.$ However, even fixing $c_1(\cald)=0$, there can be many tight contact structures.

Consider the lens space $L(p,q)$, with $p>q>0$ and $\gcd(p,q) = 1$. Generally, all lens spaces $L(p,q)$ admit Sasakian structures, but only one isotopy class can be of Sasaki type. However, Giroux \cite{Gir00} and Honda \cite{Hon00a} independently have shown that for general $p$ and $q$ there are more tight contact structures.  Write
$-{p\over
q}$ as a continued fraction:
$$-{p\over q}=r_0-{1\over r_1-{1\over r_2-\cdots {1\over r_k}}},$$
with all $r_i<-1$.
Then we have the following classification theorem for tight contact
structures on lens spaces $L(p,q)$.

\begin{theorem}[Giroux,Honda]\label{1}
There exist exactly $|(r_0+1)(r_1+1)\cdots (r_k+1)|$
tight contact structures on the lens space $L(p,q)$ up to isotopy.
Moreover, all these tight contact structures on $L(p,q)$ are Stein fillable. 
\end{theorem}

So the lens space $L(p,p-1)$ has a unique (up to co-orientation)\footnote{Uniqueness for isotopy classes will always mean up to co-orientation.} isotopy class of tight contact structures which is therefore of Sasaki type; however, $L(p,1)$ has $p-1$ isotopy classes of tight contact structures, only one of which can be of Sasaki type. It turns out that tight contact structures of Sasaki type on lens spaces are precisely the ones called `universally tight', that is their universal cover $S^3$ is also tight. See \cite{Gir00,Hon00a}. There has been much recent work on understanding the tight contact structures that can occur on Seifert fibered 3-manifolds (cf. \cite{GLS07b,LiSt07,Mas07} and references therein). However, we are mainly interested in those of Sasaki type, and these must have the contact structure $\cald$ transverse to the fibers of the Seifert fibration. This type of Seifert fibration has been studied in an unpublished work of Honda, and finalized in \cite{LiMa04} where necessary and sufficient conditions on the orbit invariants are given to admit a contact structure transverse to the fibers. An equivalent result has recently been given by Massot \cite{Mas07} who studies geodesible contact structures. It turns out that geodesible contact structures are either certain contact structures on certain torus bundles or contact structures which are transverse to a Seifert fibered structure. From the point of view of Sasakian structures, Theorem 8.1.14 of \cite{BG05} implies that if the Sasakian structure is null or of negative type (i.e. types (2) or (3) of Theorem \ref{Belgunthm}), then there is a unique Seifert fibered structure with $\cald$ transverse. However, in the positive case with the dimension of the Sasaki cone greater than one, there are infinitely many Seifert fibered structures transverse to the same contact structure $\cald.$ 

Consider Brieskorn polynomials of the form $z_0^{a_0}+z_1^{a_1}+ z_2^{a_2}$ with $a_0,a_1,a_2$ pairwise relatively prime. By Brieskorn's Graph Theorem \cite{Bri66} the link $L(a_0,a_1,a_2)$ of such a polynomial is an integral homology sphere. All such homology spheres admit natural Sasakian structures. The homology spheres can be distinquished by their Casson invariant $\grl(L(a_0,a_1,a_2))$ which as shown by Fintushel and Stern \cite{FiSt90} can be given in terms of the Hirzebruch signature of the Milnor fiber $V(a_0,a_1,a_2),$ namely 
$$\grl(L(a_0,a_1,a_2)) = \frac{\tau(V(a_0,a_1,a_2))}{8}.$$$L(a_0,a_1,a_2)$
As shown by Brieskorn (cf. Section 9.4 of \cite{BG05}), the signature can be computed by a counting argument. As concrete examples, I consider $L(6k-1,3,2)$ which gives an infinite sequence of distinct integral homology spheres. It is easy to see that $\grl(L(6k-1,3,2)) = -k.$ The case $k=1$ is the famous Poincar\'e homology sphere, while $k=2$ gives $L(11,3,2).$ I understand that it is folklore that there is exactly one isotopy class of tight contact structures on $L(5,3,2),$ although I wasn't able to find the explicit statement in the literature.  It was also shown by Ghiggini and Sch\"onenberger \cite{GhSc03} that (up to co-orientation) there is exactly 1 isotopy class of tight contact structures on $L(11,3,2)$ which must be of Sasaki type. One can also ask about the existence of tight contact structures on these homology spheres (or any other contact 3-manifold) with the orientation reversed, which we denote by $-L(a_0,a_1,a_2).$ It is known that $-L(5,3,2)$ \cite{EtHo01} does not admit any tight contact structures, but that $-L(11,3,2)$ \cite{GhSc03} admits exactly one isotopy class of tight contact structures. More generally, given any quasi-smooth weighted homogeneous polynomial $f(z_0,z_1,z_2)$, it is an interesting problem to classify the tight contact structures (up to isotopy) on both the corresponding link $L_f$ and its reversed oriented manifold $-L_f$. As far as I know this problem is still open.  In contrast to dimension 3, in dimension 7 it is known \cite{BGK05} that all homotopy spheres admit Sasakian, even Sasaki-Einstein structures for both orientations since both orientations can be represented by Brieskorn-Pham polynomials. 

\subsection{Sasaki-Einstein Structures in Dimension 5}
Most of what is known about Sasakian geometry in dimension 5 appears in Kris and my book \cite{BG05}. However, there are some updates in \cite{BGS07b} and the Ph.D dissertations of my students \cite{Cua08} and \cite{Gom08}. The reason that dimension five is so amenable to analysis stems from Smale's seminal work \cite{Sm62} on the classification of compact simply connected spin 5-manifolds. The non-spin case was later completed by Barden \cite{Bar65}, but the spin case will suffice here. Smale's classification of all closed simply 
connected $5$-manifolds that admit a spin structure is as follows. Any such manifolds must be of the form
\begin{equation}\label{smaleman}
M=kM_\infty\# M_{m_1}\#\cdots \#M_{m_n}
\end{equation} 
where $M_{\infty}= S^2\times S^3$, 
$kM_{\infty}$ is the $k$-fold connected sum of $M_{\infty}$,   
$m_i$ is a positive integer with $m_i$ dividing $m_{i+1}$ and $m_1\geq 1$, 
and where $M_{m}$ is $S^5$ if $m=1$, or a 
$5$-manifold such that $H_2(M_m,{\mathbb Z})={\mathbb Z}/m \oplus 
{\mathbb Z}/m$, otherwise. 
Here $k\in \bbn$ and $k=0$ means that there is
no $M_{\infty}$ factor at all. It will also be convenient to
use the convention $0M_{\infty}=S^5$, which is consistent
with the fact that the sphere is the identity element for the connected
sum operation. The $m_i$s can range through the positive integers.
The understanding of Sasakian geometry in dimension five owes much to the work of Koll\'ar \cite{Kol05b,Kol06a,Kol06}. In particular, he has given a complete characterization of the torsion \cite{Kol05b} for smooth Seifert bundles over a projective cyclic orbifold surface, namely
\begin{equation}\label{Koltor}
H_2(M,\bbz)=\bbz^{k}\oplus\sum_i
(\bbz_{m_i})^{2g(D_i)},
\end{equation}
where $g(D)$ denotes the genus of the branch divisor $D.$

Here we reproduce a table in \cite{BG05} that lists all simply connected 
spin 5-manifolds that can admit a Sasaki-Einstein metric. In the first column we list the 
type of manifold in terms of Smale's description (\ref{smaleman}), 
while in the second column we indicate the restrictions under which a 
Sasaki-Einstein structure is known to exist. Any Smale manifold that is not 
listed here cannot admit a Sasaki-Einstein metric.

\medskip
\begin{center}
\begin{tabular}{|c|c|}\hline
\hline 
Manifold $M$ & Conditions for SE\\
\hline\hline $kM_\infty$, $k\geq0$ & any $k$ \\
\hline
$8M_\infty\#M_m$, $m>2 $ & $m>4$ \\
\hline
$7M_\infty\# M_m$, $m>2$ & $m>2$\\
\hline
$6M_\infty\# M_m$, $m>2$ & $m>2$\\
\hline
$5M_\infty\# M_m$, $m>2$ & $m>11$\\
\hline
$4M_\infty\# M_m$, $m>2$& $m>4$\\
\hline
$3M_\infty\# M_m$, $m>2$ & $m=7,9$  or $m>10$ \\
\hline
$2M_\infty\# M_m$, $m>2$ & $m>11$ \\
\hline
$M_\infty\# M_m$, $m>2$ & $m>11$\\
\hline
$M_m$, $m>2$ & $m>2$ \\
\hline $2M_5$, $2M_4$, $4M_3$, $M_\infty\#2M_4$ & yes \\
\hline
$kM_\infty\# 2M_3 $ & $k=0$\\
\hline
$kM_\infty\# 3M_3$, $k\geq0$ & $k=0$\\
 \hline
$kM_\infty\# nM_2$, $k\geq 0$, $n>0$ & $(k,n)=(0,1)$ or $(1,n)$, $n>0$\\
\hline $kM_\infty\# M_m$, $k>8$, $2<m<12$ & \\
 \hline
\end{tabular}
\vspace{3mm}\\

\parbox{4.00in}{\small Table 1. Simply connected spin 5-manifolds admitting
 Sasaki-Einstein metrics. The right column indicates the restriction that 
ensures that the manifold
on the left carries a Sasaki-Einstein metric.}
 
\end{center}
\medskip

We now present a table with examples of links 
$L_f({\bf w},d)$ which do not admit Sasaki-Einstein metrics. This table is taken from Tables 2 and 3 of \cite{BGS07b}. 
We list the weight vector $\bfw$ and degree $d$ for each link the range for its parameters together with the underlying manifold.
Note that there are infinitely many such Sasaki-Seifert structures 
on each of the manifolds listed. 
\medskip
 
\begin{center}
\begin{tabular}{|c|c|c|}\hline
\hline
Manifold $M$ & weight vector ${\bf w}$ & $d$ \\
\hline\hline
$M_1$ & $(6,2(6l+1),2(6l+1),3(6l+1)$, $l\geq 2$ & $6(6l+1)$\\
\hline
$M_1$ & $(6,2(6l+5),2(6l+5),3(6l+5))$, $l\geq 2$ & $6(6l+5)$\\
\hline
$M_1$ & $(2,2l+1,2l+1,2l+1)$, $l\geq 2$ & $2(2l+1)$\\
\hline
$M_\infty$ & $(1,l,l,l)$, $l\geq 3$ & $2l$\\
\hline
$2M_\infty$ & $(3,2(3l+1),2(3l+1),3(3l+1))$, $l\geq 2$ & $6(3l+1)$\\
\hline
$2M_\infty$ & $(3,2(3l+2),2(3l+2),3(3l+2))$, $l\geq 2$ & $6(3l+2)$\\
\hline
$4M_\infty$ & $(1,2l,2l,3l)$, $l\geq 3$ & $6l$\\
\hline
$6M_\infty$ & $(1,3l,4l,6l)$, $l\geq 3$ & $12l$\\
\hline
$7M_\infty$ & $(1,4l+1,3l+1,2(3l+1))$, $l\geq 2$ & $4(3l+1)$ \\
\hline
$8M_\infty$ & $(1,6l,10l,15l)$, $l\geq 3$ & $30l$\\
\hline
$(k-1)M_\infty$ & $(1,2l,lk,lk)$, $l\geq 2$, $k\geq 3$ & $2lk$\\
\hline
$4M_2$ & $(2,6(2l+1),10(2l+1),15(2l+1))$, $l\geq 2$ & $30(2l+1)$\\
\hline
$M_4$ & $(4,3(2l+1),4(2l+1),4(3l+1))$, $l\geq 5$ & $12(2l+1)$ \\
\hline
$2M_\infty\# M_2$ & $(2,3(2l+1),4(2l+1),6(2l+1))$, $l\geq 2$ & $12(2l+1)$\\
\hline
$3M_\infty\# M_2$ & $(2,2(4l+3),6l+5,2(6l+5))$, $l\geq 2$ & $4(6l+5)$ \\
\hline
\end{tabular}
\vspace{3mm}\\

\parbox{4.00in}{\small Table 2. Some examples of 5-dimensional 
links whose associated space of Sasakian structures, ${\mathcal S}(\xi)$,
does not admit a Sasaki-Einstein metric.}

\end{center}
\medskip

Most, but not all, the examples in Table 2 are of BP type.
Many more examples can be obtained by considering the complete list of normal forms of weighted homogeneous polynomials given by 
Yau and Yu \cite{YaYu05}. An interesting related question posed in \cite{BGS07b} is (see also Question 11.4.1 of \cite{BG05}):

\begin{question}
Are there positive Sasakian manifolds that admit no Sasaki-Einstein or extremal metrics?
\end{question}

To end this section I will briefly comment on the case of toric Sasakian 5-manifolds of which there are many (see Section 10.4.1 of \cite{BG05}). For a somewhat more detailed description of toric Sasakian geometry, see Section \ref{torsas} below. The 5-dimensional case begins with the work of Oh \cite{Oh83} who classified all simply connected 5-manifolds that admit an effective $T^3$-action, and Yamazaki \cite{Yam01} proved that all such toric 5-manifolds also admit compatible K-contact structures. Explicitly, all simply connected toric 5-manifolds are diffeomorphic to $S^5, kM_\infty$ or
$X_\infty\#(k-1)M_\infty,$ where $k=b_2(M)\geq 1,$ that is they are precisely the simply connected 5-manifolds $M^5$ with no torsion in $H_2(M^5,\bbz).$ Here $X_\infty$ is the non-trivial $S^3$-bundle over $S^2.$ It follows by a result of Kris and me \cite{BG00b} that all such toric 5-manifolds admit a compatible Sasakian structure. There are examples of inequivalent toric Sasakian structures that are equivalent as contact structures, and even with the same contact 1-form \cite{BGO06}. Arguably the most important toric contact manifolds are those where there is a choice of contact 1-form with a compatible Sasaki-Einstein metric. Such examples were first uncovered by the physicists Gauntlett, Martelli, Sparks, and Waldram \cite{GMSW04b,GMSW04a}, and further explored in \cite{MaSp05b,MaSp06,MaSpYau05,MaSpYau06}. One interesting aspect of such Sasaki-Einstein metrics is that in most cases the Sasakian structure is irregular. The first such examples appeared in \cite{GMSW04b}. The proof of the existence of toric Sasaki-Einstein metrics on the $k$-fold connected sums $kM_\infty=k(S^2\times S^3)$ was given in \cite{FOW06} and \cite{CFO07} and will be discussed in Section \ref{torsas}.

\subsection{Sasaki-Einstein Structures in Dimension 7}
The first examples of nonhomogeneous Sasaki-Einstein metrics in dimension 7 were the 3-Sasakian metrics described in \cite{BGM93a,BGM94a}. A bit later Kris and I with Mann and Rees described quaternionic toric\footnote{This is `toric' in the quaternionc sense and is to be distinquished from the usual toric definition. So a 3-Sasakian manifold of dimension $4n+3$ is `toric' if it admits an effective $T^{n+1}$-action of 3-Sasakian automorphisms.} 3-Sasakian manifolds by symmetry reduction \cite{BGMR98}. For a complete treatment see Chapter 13 of \cite{BG05} as well as \cite{BG99}.

An effective method for constructing Sasaki-Einstein metrics in higher dimension is the join construction introduced in \cite{BG00a}. This was then exploited by Kris and I to give examples of 7-manifolds with second Betti number in the range $4\leq b_2\leq 9$ (see Theorem 5.4 of \cite{BG00a}). More recently, with L. Ornea, Kris and I \cite{BGO06} were able to use this method to prove:

\begin{theorem}\label{s2s5}
If $\gcd(l_2,\gru_2)=1$ then $S^3\star_{l_1,l_2}S^5_\bfw$ is  the
total space of the fibre bundle over $S^2$ with fibre the lens
space $L^5(l_2),$ and it admits \Se metrics. In particular, for 16
different weight vectors $\bfw$, the manifold
$S^3\star_{2,1}S^5_\bfw$ is homeomorphic to $S^2\times S^5$ and
admits Sasaki-Einstein metrics including one $10$-dimensional
family. Moreover, for three different weight vectors $\bfw$, the
manifold $S^3\star_{1,1}S^5_\bfw$ is homeomorphic to $S^2\times
S^5$ and admits Sasaki-Einstein metrics.
\end{theorem}

In \cite{BG05h} Kris and I proved the existence of positive Ricci curvature Sasakian metrics on distinct diffeomorphism types of manifolds that are homeomorphic to $2k(S^3\times S^4).$ In so doing we utilized the periodicity of certain branched covers due to Durfee and Kauffman \cite{DuKa75}. These can be  represented as links of the weighted homogeneous polynomials 
$$f_i=z_0^{2i(2k+1)}+z_1^{2k+1}+z_2^2+z_3^2+z_4^2,$$
where the diffeomorphism type is determined by $i.$ It is easy to see that the index $I_i(k)=2ik+3i+1>4,$ so the Lichnerowicz obstruction of Proposition \ref{Lich.Ob.Link.Ineq} holds. Thus, these Sasaki-Seifert structures do not admit Sasaki-Einstein metrics. Moreover, until now the existence of Sasaki-Einstein metrics on manifolds homeomorphic to $2k(S^3\times S^4)$ was unknown. Recall that a well-known result (cf. \cite{BG05}, Theorem 7.4.11) says that the odd Betti numbers of a Sasakian manifold must vanish up to the middle dimension. Thus, $(2k+1)(S^3\times S^4)$ cannot admit any Sasakian structures. 

In June of 2007, while Kris was in Lecce and I was at CIMAT in Guanajuato, Mexico, we exchanged emails concerning the recent work of Cheltsov \cite{Chel07} and the implications that it has regarding 7-dimensional Sasaki-Einstein structures. Cheltsov considers $d$-fold branch covers of certain weighted projective spaces branched over orbifold K3 surfaces of degree $d.$ The orbifold K3 surfaces that Cheltsov uses come from the well-known list of 95 codimension one K3 surfaces of Miles Reid \cite{Rei79}. This list is also repeated in \cite{Fle00} as well as Appendix B of \cite{BG05}. Of interest to Cheltsov and to us is the list given in \cite{Chel07} of anti-canonically embedded $\bbq$-Fano threefolds (this list appears in \cite{Fle00} as well) obtained from Reid's list by adding one more variable of weight one. That is one looks for hypersurfaces $X_d$ of degree $d$ in weighted projective spaces of the form $\bbc\bbp(1,w_1,w_2,w_3,w_4)$ such that $\sum_{i=1}^4w_i=d.$ Cheltsov shows that remarkably all save four of the 95 $\bbq$-Fano 3-folds satisfy a klt condition, and so admit K\"ahler-Einstein orbifold metrics. The four that fail the test are numbers 1,2,4, and 5 on the lists in \cite{Chel07} and \cite{Fle00}. By the procedure described first in \cite{BG01b} and further elaborated upon in \cite{BG05}, this means that one can construct $2$-connected 7-manifolds with Sasaki-Einstein metrics. Thus, there are 91 links representing $2$-connected 7-manifolds with Sasaki-Einstein metrics. However, by Proposition \ref{Orlikconjknown} Orlik's conjecture \ref{orl.alg} is known to hold only for polynomials of BP or OR type. I have picked out 12 on Cheltsov's list that come from BP polynomials and 2 from OR polynomials. 

\begin{theorem}\label{se7man}
There exist Sasaki-Einstein metrics on the 7-manifolds $M^7$ which can be realized as $d$-fold branched covers of $S^7$ branched along the submanifolds $k(S^2\times S^3)$.  The precise $\bbq$-Fano threefolds, the homology of $M^7$, and $k$ are listed in the following table:

\vskip 12pt
\begin{center}
\begin{tabular}{|c|c|c|}\hline
$X_d\subset \bbc\bbp(1,w_1,w_2,w_3,w_4)$ & $H_3(M^7,\bbz)$ & $k$ \\
\hline\hline
$X_6\subset\bbc\bbp(1,1,1,1,3)$ & $\bbz^{104}\oplus \bbz_2$ & $21$ \\ \hline
$X_8\subset\bbc\bbp(1,1,1,2,4)$ & $\bbz^{128}\oplus \bbz_4$ & $19$ \\ \hline
$X_{10}\subset\bbc\bbp(1,1,2,2,5)$ & $\bbz^{128}\oplus (\bbz_2)^4$ & $16$ \\ \hline
$X_{12}\subset\bbc\bbp(1,1,1,4,6)$ & $\bbz^{222}$ & $20$ \\ \hline
$X_{12}\subset\bbc\bbp(1,1,2,3,6)$ & $\bbz^{150}\oplus \bbz_{12}$ & $15$ \\ \hline
$X_{12}\subset\bbc\bbp(1,1,3,4,4)$ & $\bbz^{120}\oplus (\bbz_4)^2$ & $12$ \\ \hline
$X_{12}\subset\bbc\bbp(1,2,3,3,4)$ & $\bbz^{80}\oplus \bbz_3\oplus (\bbz_6)^2$ & $10$ \\ \hline
$X_{17}\subset\bbc\bbp(1,2,3,5,7)$ & $\bbz^{112}\oplus \bbz_{17}$ & $8$ \\ \hline
$X_{18}\subset\bbc\bbp(1,1,2,6,9)$ & $\bbz^{256}\oplus (\bbz_2)^2$ & $16$ \\ \hline
$X_{19}\subset\bbc\bbp(1,3,4,5,7)$ & $\bbz^{90}\oplus \bbz_{19}$ & $6$ \\ \hline
$X_{20}\subset\bbc\bbp(1,1,4,5,10)$ & $\bbz^{216}\oplus \bbz_5$ & $12$ \\ \hline
$X_{24}\subset\bbc\bbp(1,1,3,8,12)$ & $\bbz^{308}\oplus \bbz_3$ & $14$ \\ \hline
$X_{30}\subset\bbc\bbp(1,2,3,10,15)$ & $\bbz^{242}\oplus \bbz_2\oplus \bbz_6$ & $10$ \\ \hline
$X_{42}\subset\bbc\bbp(1,1,6,14,21)$ & $\bbz^{480}$ & $12$ \\ \hline
\end{tabular}
\vspace{3mm}\\

\end{center}
\vskip 12pt
\end{theorem}

Other examples of 7-manifolds admitting Sasaki-Einstein metrics were given in Section 11.7.3 of \cite{BG05}. For example, the links $L(p,p,p,p,p)$ of Fermat hypersurfaces with $2\leq p\leq 4$ are well known to admit Sasaki-Einstein metrics. These give the Stiefel manifold $V_2(\bbr^5)$ of 2-frames in $\bbr^5$ for $p=2,$ and for $p=3,4$ they give 2-connected 7-manifolds $M^7$ with homology $H_3(M^7,\bbz)$ equal to $\bbz^{10}\oplus \bbz_3$, and $\bbz^{60}\oplus \bbz_4$ for $p=3$ and $4$, respectively. We also consider $4m$-fold branched covers of Fermat hypersurfaces represented by the BP polynomial $f=z_0^{4m}+z_1^4+z_2^4+z_3^4+z_4^4.$ It is easy to see that the conditions of Equation \ref{bgkest} are satisfied when $m\geq 4,$ so these links $L(4m,4,4,4,4)$ admit Sasaki-Einstein metrics. The corresponding 7-manifolds $M^7_m$ appear to have homology $H_3(M^7_m,\bbz)=\bbz^{60}\oplus \bbz_{4m}\oplus (\bbz_m)^{20}.$ I have not proven this expression rigorously, but we have checked it on our computer program to be correct for over 10 cases. Notice that, as with the Cheltsov examples, these are all $d$-fold branched covers of $S^7$ branched over $k(S^2\times S^3),$ with $k=21$ here. It is interesting to note that the klt estimates of \cite{BGK05} and those of Cheltsov appear to be complimentary in nature. 

Another approach to Sasaki-Einstein structures on 7-manifolds would entail a systematic study of the $S^1$ orbibundles over the 1936 Fano 4-folds found by Johnson and Koll\'ar \cite{JoKo01b} that admit K\"ahler-Einstein metrics. This has already been done for the case of rational homology 7-spheres in \cite{BGN02a}. Kris and I had planned a systematic study of Sasaki-Einstein 7-manifolds; however, in contrast to the case of 5-manifolds, this is not so straightforward. First, in dimension 7 there is no nice classification of 2-connected 7-manifolds such as the Smale-Barden classification of simply connected 5-manifolds. Second the  realization of 7-manifolds in terms of links of weighted homogeneous hypersurfaces has so far yielded few results for the torsion-free case. Indeed, the only cases where Sasaki-Einstein metrics are known to exist on 7-manifolds of the form $2k(S^3\times S^4)$ are the two announced in Theorem \ref{se7man} above, namely $222(S^3\times S^4)$ and $480(S^3\times S^4)$. Third, the veracity of the Orlik Conjecture \ref{orl.alg} is still in doubt for general weighted homogeneous polynomials in the 7-dimensional case.

\begin{example}{\it Moduli of Sasaki-Einstein metrics}. 
Consider the link $L_f(\bfw,d)$ of the Brieskorn-Pham polynomial $f= f_{12}(z_0,z_1,z_2) +z_3^3+z_4^2$ of degree $d=12$ and weight vector $\bfw=(1,1,1,4,6).$ Here $f_{12}$ is a nondegenerate polynomial of degree $12$ in the specified variables. By Cheltsov's results this is one of the examples referred to in Theorem \ref{se7man} and it is homeomorphic to $222(S^3\times S^4).$ According to Theorem 5.5.7 of \cite{BG05} the number of moduli equals
$$2\bigl(h^0(\bbc\bbp(\bfw),\calo(12))-\sum_ih^0(\bbc\bbp(\bfw),\calo(w_i))\bigr) =2(184-51)=266.$$
So for some (as yet undetermined) diffeomorphism class on $222(S^3\times S^4)$, there exists a Sasaki-Einstein metric depending on 266 parameters.
\end{example}

\section{Toric Sasakian Geometry}\label{torsas}

It seems to have first been noticed in \cite{BM93} that toric contact manifolds split into two types exhibiting strikingly different behavior. Assuming $\dim M>3$ there are those where the torus action is free which are essentially torus bundles over spheres, and those where the torus action is not free. It is only these later that display the familiar convexity properties characteristic of symplectic toric geometry. A complete classification was given by Lerman \cite{Ler02a}. Moreover, it was proven a bit earlier by Kris and me \cite{BG00b} that toric contact structures of the second type all admit compatible Sasakian metrics, and so are Sasakian. I shall only consider this type of toric contact manifold in what follows. We had dubbed these {\it contact toric structures of Reeb type} since they are characterized by the fact that the Reeb vector field lies in the Lie algebra of the torus. We also say that the action is of {\it Reeb type}. Before embarking on the toric case let us consider the case of a compact Sasakian manifold $M^{2n+1}$ with the action of a torus $T^k$ of dimension $1\leq k\leq n+1.$ This action is automatically of Reeb type. The toric case is $k=n+1.$ I refer to Section 8.4 of \cite{BG05} and references therein for details and further discussion. 

\subsection{The Moment Cone}
Let $\cald^o$ be the annihilator of $\cald$ in $T^*M.$ This is a real line bundle called the {\it contact line bundle}, and a choice of contact 1-form $\eta$ gives a trivialization of $\cald^o$ as well as giving an orientation of $\cald^o$ which splits $\cald^o$ as $\cald^o_+\cup \cald_-^o.$  Given an action of a k-torus $T^k$ of Reeb type on $M,$ one has a moment map $\mu:\cald^o_+\ra{1.5} \gt^*_k,$ where $\gt^*_k$ denotes the dual of the Lie algebra $\gt_k$ of $T^k,$ and $\cald^o_+$ denotes the positive part of $\cald^o.$ The image $\mu(\cald^o_+)$ is a convex rational polyhedral cone (without the cone point) in $\gt^*_k$ and coincides with the image of the symplectic moment map from the K\"ahler cone $C(M)$ with symplectic form $d(r^2\eta)$ for any choice of contact form $\eta.$ Adding the cone point we get the {\it moment cone} $C(\mu)=\mu(\cald_+^o)\cup \{0\}.$ Note that a choice of contact form  is a section of $\cald^o_+,$ and for any $T^k$-invariant section $\eta,$ the moment cone can be realized as
\begin{equation}\label{momcone}
C(\mu_\eta)=\{r\grg\in \gt_k^*~|~\grg\in \mu_\eta(M)~,~t\in [0,\infty)\}=C(\mu). 
\end{equation}
The dual cone $C(\mu)^*$ is the set of all $\grt\in \gt_k$ such that $\langle\grt,\grg\rangle\geq 0$ for all $\grg\in C(\mu);$ so its interior is precisely the Sasaki cone $\gt_k^+.$ Alternatively, the moment cone $C(\mu)$ can be described by
\begin{equation}\label{momcone2}
C(\mu)=\{y\in \gt_{k}^*~|~l_i(y)=\langle y, \grl_i\rangle\geq 0,~i=1,\ldots d\},
\end{equation}
where the $\grl_i$ are vectors spanning the integral lattice $\bbz_{T^{k}}=\ker\{\exp:\gt_{k}\ra{1.5} T^{k}\},$ and $d$ is the number of facets (codimension one faces) of $C(\mu).$ A choice of contact form $\eta$ whose Reeb vector field lies in $\gt_k^+$ cuts $C(\mu)$ by a hyperplane, called the {\it characteristic hyperplane} \cite{BG00b}, giving a simple polytope.

We now specialize to the toric case when $k=n+1$ for which I follow \cite{MaSpYau05,FOW06}. This uses a well-known procedure of Guillemin \cite{Gui94b} which constructs natural toric K\"ahler metrics on compact toric K\"ahler manifolds arising from Delzant's combinatorical description. This construction works equally well for both toric K\"ahler cones and toric K\"ahler orbifolds \cite{Abr01}, and hence toric Sasakian structures as shown in \cite{MaSpYau05} (see also Theorem 8.5.11 of \cite{BG05}). 
Given a toric contact structure of Reeb type one can obtain a `canonical Sasakian structure' from a potential $G$ as follows.  Let us also define
\begin{equation}\label{momcone3}
l_\xi(y)= \langle y,\xi\rangle, \qquad l_\infty = \sum_{i=1}^d \langle y,\grl_i\rangle.
\end{equation}
 Then for each $\xi$ in the Sasaki cone $\gt_{n+1}^+$ we have a potential:
\begin{equation}\label{torpot}
G_\xi= \frac{1}{2}\sum_{i=1}^dl_i(y)\log l_i(y) +\frac{1}{2}l_\xi(y)\log l_\xi(y) -\frac{1}{2}l_\infty(y)\log l_\infty(y).
\end{equation}
The K\"ahler cone metric $g_{C(M)}$ of Equation (\ref{conemetric}) can then be computed from this potential \cite{Abr03,MaSpYau05} in the so-called {\it symplectic coordinates}, viz.
\begin{equation}\label{transtoricmetric}
g_{C(M)} = \sum_{i,j}\bigl((G_\xi)_{ij}dy_idy_j + G_\xi^{ij}d\phi_id\phi_j\bigr),
\end{equation}
where
\begin{equation}\label{derpot}
(G_\xi)_{ij}=\frac{\partial^2 G}{\partial y_i\partial y_j}
\end{equation}
and $G_\xi^{ij}$ denotes the inverse matrix of $G_{ij}.$ Here the coordinates $\phi_i$ are coordinates on the torus $T^{n+1},$ and the vector field given by $\xi=\sum_{i=0}^{n}\xi^i\frac{\partial}{\partial \phi_i}$ with $\xi^i$ constants restricts to the Reeb vector field of a Sasakian structure on the manifold $M$. 

We now want to impose the condition $c_1(\cald)=0$ as an integral class\footnote{Everything goes through with the weaker condition that $c_1(\cald)_\bbr=0,$ since this is all that is required to obtain the condition $c_1(\calf_\xi)=a[d\eta]_B.$ Indeed, it is this more general case that is treated in \cite{FOW06}. In this case the form $\grO$ is not a section of the canonical line bundle on $C(M)$, but rather some power of it. But we impose the stronger condition for ease of exposition as the generalization is straightforward.}. From Remark \ref{Goren} we see that this implies the existence of a nowhere vanishing holomorphic section $\grO$ of $\grL^{n+1,0}T^*C(M).$ Notice that $\grO$ is not unique, but defined up to a nowhere vanishing holomorphic function on $C(M).$ This function can be chosen so that 
\begin{equation}\label{Omegaeqn}
\pounds_\xi\grO =i(n+1)\grO.
\end{equation}
The Ricci form $\grr_C$ of the cone $C(M)$ can be computed from the Legendre transform $F$ of the potential $G.$ The K\"ahler potential on $C(M)$ is  
\begin{equation}\label{Legendretrans}
F = \sum_iy_i\frac{\partial G}{\partial y_i}-G.
\end{equation}
and then the Ricci form becomes 
\begin{equation}\label{Riccicone}
\grr_C=-i\partial \bar{\partial}\log \det F_{ij},
\end{equation}
where 
$$F_{ij}=\frac{\partial F}{\partial x_i\partial x_j}, \qquad x_i=\frac{\partial G}{\partial y_i}.$$
So imposing the condition $c_1(TC(M))=c_1(\cald)=0,$ implies that there are constants $\grg_0,\ldots,\grg_n$ and a basic function $h$ such that 
\begin{equation}\label{Ricciflat}
\log\det F_{ij}= -2\sum_i\grg_ix_i -h.
\end{equation}
This makes the Hermitian metric $e^h\det F_{ij}$ the flat metric on the canonical line bundle $K_{C(M)}.$ Then one can find complex coordinates $\bfz=(z_0,\ldots,z_n)$ such that the holomorphic $(n+1)$-form $\grO$ takes the form
\begin{equation}\label{holoform2}
\grO=e^{-\boldsymbol{\grg}\cdot \bfz}dz_0\wedge\cdots \wedge dz_n
\end{equation}
where $\boldsymbol{\grg}=(\grg_0,\ldots,\grg_n)$ is a primitive element of the integral lattice which by explicitly computing in terms of the potential $G_\xi$ one sees that $\boldsymbol{\grg}$ is uniquely determined by the moment cone $C(\mu)$ as 
\begin{equation}\label{gammaeqn}
\langle \grl_j,\boldsymbol{\grg}\rangle=-1, \qquad j=1,\ldots d.
\end{equation} 
Then using Equation (\ref{Omegaeqn}) one finds that the Reeb vector field $\xi$ should satisfy
\begin{equation}\label{specialReeb}
\langle \xi,\boldsymbol{\grg}\rangle= -(n+1).
\end{equation}
This cuts the Sasaki cone $\grk(\cald,J)$ by a hyperplane giving a convex polytope in $\grS\subset \grk(\cald,J)$. Martelli, Sparks, and Yau \cite{MaSpYau06} (see also \cite{Spa08}) showed that, assuming existence, a Sasaki-Einstein metric is the unique critical point of the restriction of the volume functional to $\grS.$ Explicitly, we consider the functional $\calv:\grS\ra{1.5} \bbr^+$ defined by
\begin{equation}\label{volfun}
\calv(\xi)=\int_Md{\rm vol}(\cals).
\end{equation}
By the convexity of $\grS$, there is a unique critical point of $\calv,$
and Futaki, Ono, and Wang proved that this unique minimum $\xi_0$ has vanishing Sasaki-Futaki invariant $\gS\gF_{\xi_0},$ Equation (\ref{SFinv}), (this was shown in the quasi-regular case in \cite{MaSpYau06}).

\subsection{The Monge-Amp\`ere Equation}
The Sasaki-Futaki invariant $\gS\gF_{\xi}$ only depends on the deformation class ${\mathcal S}(\xi,\bar{J})$ and not on any particular Sasakian structure $\cals\in {\mathcal S}(\xi,\bar{J})$ \cite{BGS06}. However, its vanishing is only a necessary condition for the existence of a Sasaki-Einstein metric. In order to give a sufficient condition, one needs to find a global solution (i.e. for all $t\in [0,1])$ of the a well-known Monge-Amp\`ere problem. This was done by Futaki, Ono, and Wang \cite{FOW06} in the toric case discussed above when $c_1(\cald)=0,$ and $c_1^B(\calf_\xi)$ is positive. The latter conditions imply that $c_1^B(\calf_\xi)=\frac{a}{2\pi}[d\eta]_B$ with $a$ positive. Their result is the Sasaki analogue of a result of Wang and Zhu \cite{WaZh04} who proved the existence of K\"ahler-Ricci solitons on any compact toric Fano manifold together with uniqueness up to holomorphic automorphisms, and then proved that a K\"ahler-Ricci soliton becomes a K\"ahler-Einstein metric if and only if the Futaki invariant vanishes. However, in the Sasakian case there is some more freedom by deforming within the Sasaki cone. As in the K\"ahler case, Futaki, Ono, and Wang consider ``Sasaki-Ricci solitons''. 

\begin{definition}\label{sasricsol}
A {\bf Sasaki-Ricci soliton} is a Sasakian structure $\cals=(\xi,\eta,\Phi,g)$ that satisfies
$$\grr-2nd\eta=\pounds_Xd\eta$$
for some $X\in \gh(\xi,\bar{J}),$ and where $\grr$ is the Ricci form of $g$.  
\end{definition}

Then we say that $\cals=(\xi,\eta,\Phi,g)$ is a Sasaki-Ricci soliton {\it with respect to} $X.$
Note that if also $\pounds_Xd\eta=0$, a Sasaki-Ricci soliton is a Sasaki-Einstein structure. In particular, this holds if $X\in \ga\gu\gt(\cals).$

Now $c_1^B(\calf_\xi)=a[d\eta]_B$ implies the existence of a basic function $h$ such that 
\begin{equation}\label{transverseddbar}
\grr=(a-2)d\eta +d_Bd_B^ch.
\end{equation}
In analogy with the K\"ahlerian case studied by Tian and Zhu \cite{TiZh02}, Futaki, Ono, and Wang introduce a generalized Sasaki-Futaki invariant as follows: Given a nontrivial transversely holomorphic vector field $X$, the transverse Hodge Theorem gaurentees the existence of a smooth basic complex-valued function $\theta_X$ on $M,$ called a {\it holomorphy potential}, satisfying the two equations
\begin{equation}\label{basfn}
 X\hook d\eta = \frac{i}{2\pi}\bar{\partial}\theta_X, \qquad \int_M\theta_X e^{h} \eta\wedge (d\eta)^n =0.
\end{equation}
Then the {\it generalized Sasaki-Futaki invariant} is
\begin{equation}\label{gensasfu}
\gF\gG_X(v)=-\int_M  \theta_v e^{\theta_X}\eta\wedge (d\eta)^n.
\end{equation}
The invariant $\gF\gG_X$ obstructs the existence of Sasaki-Ricci solitons in the same way that the Sasaki-Futaki invariant $\gS\gF$ obstructs the existence of Sasaki-Einstein metrics. Indeed, when $X=0$ it reduces to the Sasaki-Futaki invariant $\gS\gF$. Note that if $\cals=(\xi,\eta,\Phi,g)$ is a Sasaki-Ricci solition, then $\theta_X=h +c$ for some constant $c$, and then $\gF\gG_X(v)=0$ for all $v\in \gh(\xi,\bar{J}).$ As in \cite{TiZh02} one can see that there exists an $X\in \gh(\xi,\bar{J})$ such that $\gF\gG_X$ vanishes on the reductive subalgebra of $\gh(\xi,\bar{J})$.

The existence of Sasaki-Ricci solitons then boils down to proving the existence of solutions for all $t\in [0,1]$ to the Monge-Amp\`ere equation: 
\begin{equation}\label{MAeqn}
\frac{\det (g^T_{i\bar{j}} + \varphi_{i\bar{j}})}{\det (g^T_{i\bar{j}})} = \exp( - t (2m+2)\varphi - \theta_X - X\varphi + h).
\end{equation}
By the continuity method this amounts to finding a $C^0$ estimate for the function $\varphi.$ The proof is quite technical, to which I refer to Section 7 of \cite{FOW06}. I rephase this important result of \cite{FOW06} as 

\begin{theorem}\label{FOW}
Let $(M,\cald)$ be a compact toric contact manifold of Reeb type with $c_1(\cald)=0.$ Then there exists a Sasakian structure $\cals=(\xi,\eta,\Phi,g)$ in the Sasaki cone $\grk(\cald,J)$ such that $g$ is a Sasaki-Einstein metric.
\end{theorem}

As an immediate application of this theorem we have:

\begin{example}\label{cpblowups}
Let $N_k$ denote $\bbc\bbp^n$ blown-up at $k$ generic points, where $k=1,\ldots,n,$ so that $N_k=\bbc\bbp^n\# k\overline{\bbc\bbp}^n.$ The Lie algebra $\gh(N_k)$ of infinitesimal automorphisms of the complex structure is not reductive, so $N_k$ does not admit K\"ahler-Einstein metrics. Let $M^{2n+1}_k$ denote the total space of the unit sphere bundle in the anti-canonical line bundle $K^{-1}_{N_k}.$ This has a naturally induced toric Sasakian structure, but does not admit a Sasaki-Einstein metric. However, by Theorem \ref{FOW} there is a Reeb vector field $\xi$ within the Sasaki cone $\grk(\cald,J)$ whose associated Sasakian structure is Sasaki-Einstein. Generically, such Sasaki-Einstein structures will be irregular as seen in \cite{GMSW04a}.
\end{example}

In dimension 5 one has \cite{CFO07}

\begin{corollary}\label{toric5se}
For each positive integer $k$ the manifolds $k(S^2\times S^3)$ admit infinite families of toric Sasaki-Einstein structures.
\end{corollary}

In \cite{CFO07} the Sasakian analogue of a the well-known uniqueness result of Bando and Mabuchi \cite{BaMa85} for Fano manifolds was proven. Explicitly, Cho, Futaki, and Ono prove that if $c_1(\cald)=0$ and $c_1^B(\calf_\xi)>0$, then a Sasaki-Einstein metric is unique up to transverse holomorphic transformations. However, there is another uniqueness question that was posed in \cite{BG05} as Conjecture 11.1.18. This conjectures that Sasaki-Einstein metrics can occur for at most one Reeb vector field in the Sasaki cone.

If one replaces Sasaki-Einstein by Sasaki-Ricci soliton, then the conjecture does not hold as seen by considering the circle bundle $M^{2n+1}$ over $\bbc\bbp^{n+1}\#\overline{\bbc\bbp^{n+1}}$ defined by the primitive cohomology class of its anti-canonical line bundle. Now the base is toric so the manifold $M^{2n+1}$ is a toric contact manifold with a compatible Sasakian structure $\cals_0.$ Moreover, it is well-known that $\bbc\bbp^{n+1}\#\overline{\bbc\bbp^{n+1}}$ does not admit a K\"ahler-Einstein metric. Thus by Theorem \ref{FOW}, there is another Reeb vector field $\xi$ on $M^{2n+1}$ whose Sasakian structure $\cals=(\xi,\eta,\Phi,g)$ is Sasaki-Einstein. But according to Tian and Zhu \cite{TiZh02} $\bbc\bbp^{n+1}\#\overline{\bbc\bbp^{n+1}}$ admits a unique K\"ahler-Ricci soliton. Hence, the circle bundle $M^{2n+1}$ admits a Sasaki-Ricci soliton which can be taken as the Sasakian structure $\cals_0=(\xi_0,\eta_0,\Phi_0,g_0).$ So the toric contact manifold $(M^{2n+1},\cald,T^{n+1})$ admits two distinct Sasaki-Ricci solitions one of which is Sasaki-Einstein.

Dropping the condition $c_1(\cald)$ one can ask whether Theorem \ref{FOW} holds with Sasaki-Einstein being replaced by constant scalar curvature, or perhaps more  generally, a strongly extremal Sasakian structure. This is currently under study.

\subsection{Towards a Classification of Toric Sasaki-Einstein Structures}
By combining the Existence Theorem \ref{FOW} of Futaki, Ono, and Wang with the Delzant type theorem of Lerman \cite{Ler02}, it is possible to give a classification of toric Sasaki-Einstein structures along the lines that was done for toric 3-Sasakian structures \cite{BGMR98,Bie99} (see Section 7 of Chapter 13 in \cite{BG05}). This also has the advantage of giving an explicit description in terms of matrices.

Consider the symplectic reduction of $\bbc^n$ (or
equivalently the Sasakian reduction of $S^{2n-1}$) by a $k$-dimensional torus $T^k.$ Every
complex representation of $T^k$ on $\bbc^{n}$ can be described
by an exact sequence
$$0\ra{1.3}T^k\fract{f_\Omega}{\ra{1.3}} T^{n}\ra{1.3} T^{n-k}\ra{1.3}0\, .$$
The monomorphism $f_\grO$ can be represented by the diagonal matrix
$$f_\Omega(\tau_1,\ldots,\tau_k)={\rm diag}\Biggl(
\displaystyle{\prod_{i=1}^k\tau_i^{a^i_1}},\ldots,
\displaystyle{\prod_{i=1}^k\tau_i^{a^i_{n}}}\Biggr)\, ,$$ where
$(\tau_1,..,\tau_k)\in S^1\times\cdots\times S^1=T^k$ are the
complex coordinates on $T^k,$ and $a^i_\gra\in \bbz$  are the
coefficients of the integral {\it weight} matrix
$\Omega\in\calm_{k,n}(\bbz),$ where $\calm_{k,n}(R)$ denotes the space of $k\times n$ matrices over the commutative ring $R.$ The following proposition was given in \cite{BG05}:

\begin{proposition}\label{toric.Sasakian.quotients}
Let $X(\Omega)=(\bbc^n\setminus{\bf0})/\!\!/T^k(\Omega)$ denote
the K\"ahler quotient of the standard flat K\"ahler structure on
$(\bbc^n\setminus{\bf0})$  by the weighted Hamiltonian
$T^k$-action with an integer weight matrix $\Omega.$ Consider the
K\"ahler moment map
\begin{equation}
\mu^i_\Omega(\bfz)=
\sum_{\alpha=1}^{n}a^i_{\alpha}|z_\alpha|^2,\qquad i=1,\ldots,k\,.
\end{equation}
If all minor $k\times k$ determinants of \ $\Omega$ are non-zero
then there is a choice of $\grO$ such that $X(\Omega)=C(Y(\Omega))$ is a cone on a compact Sasakian
orbifold $Y(\Omega)$ of dimension $2(n-k)-1$ which is the Sasakian
reduction of the standard Sasakian structure on $S^{2n-1}.$ In
addition, the projectivization of $X(\Omega)$ defined by
$\calz(\Omega)=X(\Omega)/\bbc^*$ is a K\"ahler reduction of the
complex projective  space $\bbc\bbp^{n-1}$ by a Hamiltonian
$T^k$-action defined by $\Omega$ and it is the transverse space of
the Sasakian structure on $Y(\Omega)$ induced by the quotient. If
\begin{equation}\label{CY.condition}
\sum_\alpha a^i_\alpha=0,\qquad \forall~ i=1,\ldots,k
\end{equation}
then $c_1(X(\Omega))=c_1(\cald)=0.$ In particular, the orbibundle
$Y(\Omega)\ra{1.2}\calz(\Omega)$ is anticanonical. Moreover, the
cone $C(Y(\Omega))$, its Sasakian base $Y(\Omega)$, and the
transverse space $\calz(\Omega)$ are all toric orbifolds.
\end{proposition}

Of particular interest are the conditions under which the orbifold $Y(\grO)$ is a smooth manifold. A complete study of this classification is currently in progress.
The special case $n=4$ and $k=1$ gives all toric Sasaki-Einstein structures on $S^2\times S^3.$ A subset of these was considered first in \cite{GMSW04a} giving the first examples of irregular Sasaki-Einstein structures. The general case was described independently in \cite{MaSp05b,CLPP05}. A completely different approach to toric Sasaki-Einstein geometry was given in \cite{coe06,coe07}.

\newcommand{\etalchar}[1]{$^{#1}$}
\def\cprime{$'$} \def\cprime{$'$} \def\cprime{$'$} \def\cprime{$'$}
  \def\cprime{$'$} \def\cprime{$'$} \def\cprime{$'$} \def\cprime{$'$}
  \def\cdprime{$''$}
\providecommand{\bysame}{\leavevmode\hbox to3em{\hrulefill}\thinspace}
\providecommand{\MR}{\relax\ifhmode\unskip\space\fi MR }
\providecommand{\MRhref}[2]{%
  \href{http://www.ams.org/mathscinet-getitem?mr=#1}{#2}
}
\providecommand{\href}[2]{#2}


\begin{thebibliography}{GMSW04b}

\bibitem[Abe77]{Abe77}
K.~Abe, \emph{On a generalization of the {H}opf fibration. {I}. {C}ontact
  structures on the generalized {B}rieskorn manifolds}, T\^ohoku Math. J. (2)
  \textbf{29} (1977), no.~3, 335--374. \MR{0464253 (57 \#4187)}

\bibitem[ABK{\etalchar{+}}94]{ABKLR94}
B.~Aebischer, M.~Borer, M.~K{\"a}lin, Ch. Leuenberger, and H.~M. Reimann,
  \emph{Symplectic geometry}, Progress in Mathematics, vol. 124, Birkh\"auser
  Verlag, Basel, 1994, An introduction based on the seminar in Bern, 1992.
  \MR{MR1296462 (96a:58082)}

\bibitem[Abr01]{Abr01}
M.~Abreu, \emph{K\"ahler metrics on toric orbifolds}, J. Differential Geom.
  \textbf{58} (2001), no.~1, 151--187. \MR{1895351 (2003b:53046)}

\bibitem[Abr03]{Abr03}
Miguel Abreu, \emph{K\"ahler geometry of toric manifolds in symplectic
  coordinates}, Symplectic and contact topology: interactions and perspectives
  (Toronto, ON/Montreal, QC, 2001), Fields Inst. Commun., vol.~35, Amer. Math.
  Soc., Providence, RI, 2003, pp.~1--24. \MR{MR1969265 (2004d:53102)}

\bibitem[Aka87]{Aka87}
Takao Akahori, \emph{A new approach to the local embedding theorem of
  {CR}-structures for {$n\geq 4$} (the local solvability for the operator
  {$\overline\partial\sb b$} in the abstract sense)}, Mem. Amer. Math. Soc.
  \textbf{67} (1987), no.~366, xvi+257. \MR{MR888499 (88i:32027)}

\bibitem[Bai57]{Bai57}
W.~L. Baily, \emph{On the imbedding of {$V$}-manifolds in projective space},
  Amer. J. Math. \textbf{79} (1957), 403--430. \MR{20 \#6538}

\bibitem[Bar65]{Bar65}
D.~Barden, \emph{Simply connected five-manifolds}, Ann. of Math. (2)
  \textbf{82} (1965), 365--385. \MR{32 \#1714}

\bibitem[Bau00]{Bau00}
H.~Baum, \emph{Twistor and {K}illing spinors in {L}orentzian geometry}, Global
  analysis and harmonic analysis (Marseille-Luminy, 1999), S\'emin. Congr.,
  vol.~4, Soc. Math. France, Paris, 2000, pp.~35--52. \MR{2002d:53060}

\bibitem[Bel01]{Bel01}
F.~A. Belgun, \emph{Normal {CR} structures on compact 3-manifolds}, Math. Z.
  \textbf{238} (2001), no.~3, 441--460. \MR{2002k:32065}

\bibitem[BER99]{BER99}
M.~Salah Baouendi, Peter Ebenfelt, and Linda~Preiss Rothschild, \emph{Real
  submanifolds in complex space and their mappings}, Princeton Mathematical
  Series, vol.~47, Princeton University Press, Princeton, NJ, 1999.
  \MR{MR1668103 (2000b:32066)}

\bibitem[BG99]{BG99}
C.~P. Boyer and K.~Galicki, \emph{3-{S}asakian manifolds}, Surveys in
  differential geometry: essays on Einstein manifolds, Surv. Differ. Geom., VI,
  Int. Press, Boston, MA, 1999, pp.~123--184. \MR{2001m:53076}

\bibitem[BG00a]{BG00b}
\bysame, \emph{A note on toric contact geometry}, J. Geom. Phys. \textbf{35}
  (2000), no.~4, 288--298. \MR{2001h:53124}

\bibitem[BG00b]{BG00a}
\bysame, \emph{On {S}asakian-{E}instein geometry}, Internat. J. Math.
  \textbf{11} (2000), no.~7, 873--909. \MR{2001k:53081}

\bibitem[BG01]{BG01b}
\bysame, \emph{New {E}instein metrics in dimension five}, J. Differential Geom.
  \textbf{57} (2001), no.~3, 443--463. \MR{2003b:53047}

\bibitem[BG05]{BoGa05a}
\bysame, \emph{Sasakian geometry, hypersurface singularities, and {E}instein
  metrics}, Rend. Circ. Mat. Palermo (2) Suppl. (2005), no.~75, suppl., 57--87.
  \MR{2152356 (2006i:53065)}

\bibitem[BG06a]{BG06b}
\bysame, \emph{Einstein metrics on rational homology spheres}, J. Differential
  Geom. \textbf{74} (2006), no.~3, 353--362. \MR{MR2269781 (2007j:53049)}

\bibitem[BG06b]{BG05h}
\bysame, \emph{Highly connected manifolds with positive {R}icci curvature},
  Geom. Topol. \textbf{10} (2006), 2219--2235 (electronic). \MR{MR2284055
  (2007k:53057)}

\bibitem[BG08]{BG05}
Charles~P. Boyer and Krzysztof Galicki, \emph{Sasakian geometry}, Oxford
  Mathematical Monographs, Oxford University Press, Oxford, 2008.
  \MR{MR2382957}

\bibitem[BGK05]{BGK05}
C.~P. Boyer, K.~Galicki, and J.~Koll\'ar, \emph{Einstein metrics on spheres},
  Ann. of Math. (2) \textbf{162} (2005), no.~1, 557--580. \MR{2178969
  (2006j:53058)}

\bibitem[BGKT05]{BGKT05}
C.~P. Boyer, K.~Galicki, J.~Koll{\'a}r, and E.~Thomas, \emph{Einstein metrics
  on exotic spheres in dimensions 7, 11, and 15}, Experiment. Math. \textbf{14}
  (2005), no.~1, 59--64. \MR{2146519 (2006a:53042)}

\bibitem[BGM93]{BGM93a}
C.~P. Boyer, K.~Galicki, and B.~M. Mann, \emph{Quaternionic reduction and
  {E}instein manifolds}, Comm. Anal. Geom. \textbf{1} (1993), no.~2, 229--279.
  \MR{95c:53056}

\bibitem[BGM94]{BGM94a}
\bysame, \emph{The geometry and topology of {$3$}-{S}asakian manifolds}, J.
  Reine Angew. Math. \textbf{455} (1994), 183--220. \MR{96e:53057}

\bibitem[BGM06]{BGM06}
C.~P. Boyer, K.~Galicki, and P.~Matzeu, \emph{On eta-{E}instein {S}asakian
  geometry}, Comm. Math. Phys. \textbf{262} (2006), no.~1, 177--208.
  \MR{2200887 (2007b:53090)}

\bibitem[BGMR98]{BGMR98}
C.~P. Boyer, K.~Galicki, B.~M. Mann, and E.~G. Rees, \emph{Compact
  {$3$}-{S}asakian {$7$}-manifolds with arbitrary second {B}etti number},
  Invent. Math. \textbf{131} (1998), no.~2, 321--344. \MR{99b:53066}

\bibitem[BGN02a]{BGN02a}
C.~P. Boyer, K.~Galicki, and M.~Nakamaye, \emph{Einstein metrics on rational
  homology {$7$}-spheres}, Ann. Inst. Fourier (Grenoble) \textbf{52} (2002),
  no.~5, 1569--1584. \MR{2003j:53056}

\bibitem[BGN02b]{BGN02b}
\bysame, \emph{Sasakian-{E}instein structures on {$9\#(S^2\times S^3)$}},
  Trans. Amer. Math. Soc. \textbf{354} (2002), no.~8, 2983--2996 (electronic).
  \MR{2003g:53061}

\bibitem[BGN03a]{BGN03a}
\bysame, \emph{On positive {S}asakian geometry}, Geom. Dedicata \textbf{101}
  (2003), 93--102. \MR{MR2017897 (2005a:53072)}

\bibitem[BGN03b]{BGN03c}
\bysame, \emph{On the geometry of {S}asakian-{E}instein 5-manifolds}, Math.
  Ann. \textbf{325} (2003), no.~3, 485--524. \MR{2004b:53061}

\bibitem[BGO07]{BGO06}
C.~P. Boyer, K.~Galicki, and L.~Ornea, \emph{{Constructions in Sasakian
  Geometry}}, Math Z. \textbf{257} (2007), no.~4, 907--924.

\bibitem[BGS08a]{BGS06}
C.~P. Boyer, K.~Galicki, and S.~Simanca, \emph{Canonical {S}asakian metrics},
  Commun. Math. Phys. \textbf{279} (2008), 705--733.

\bibitem[BGS08b]{BGS07b}
\bysame, \emph{The {S}asaki cone and extremal {S}asakian metrics}, Proceedings
  on the Conference on Riemannian Topology, to appear (K.~Galicki and
  S.~Simanca, eds.), Birkhauser, 2008.

\bibitem[Bie99]{Bie99}
R.~Bielawski, \emph{Complete hyper-{K}\"ahler {$4n$}-manifolds with a local
  tri-{H}amiltonian {$\bold R^n$}-action}, Math. Ann. \textbf{314} (1999),
  no.~3, 505--528. \MR{2001c:53058}

\bibitem[BM87]{BaMa85}
S.~Bando and T.~Mabuchi, \emph{Uniqueness of {E}instein {K}\"ahler metrics
  modulo connected group actions}, Algebraic geometry, Sendai, 1985, Adv. Stud.
  Pure Math., vol.~10, North-Holland, Amsterdam, 1987, pp.~11--40.
  \MR{89c:53029}

\bibitem[BM93]{BM93}
A.~Banyaga and P.~Molino, \emph{G\'eom\'etrie des formes de contact
  compl\`etement int\'egrables de type toriques}, S\'eminaire Gaston Darboux de
  G\'eom\'etrie et Topologie Diff\'erentielle, 1991--1992 (Montpellier), Univ.
  Montpellier II, Montpellier, 1993, pp.~1--25. \MR{94e:53029}

\bibitem[Bri66]{Bri66}
E.~Brieskorn, \emph{Beispiele zur {D}ifferentialtopologie von
  {S}ingularit\"aten}, Invent. Math. \textbf{2} (1966), 1--14. \MR{34 \#6788}

\bibitem[Cal82]{Cal82}
E.~Calabi, \emph{Extremal {K}\"ahler metrics}, Seminar on Differential
  Geometry, Ann. of Math. Stud., vol. 102, Princeton Univ. Press, Princeton,
  N.J., 1982, pp.~259--290. \MR{83i:53088}

\bibitem[CFO08]{CFO07}
Koji Cho, Akito Futaki, and Hajime Ono, \emph{Uniqueness and examples of
  compact toric {S}asaki-{E}instein metrics}, Comm. Math. Phys. \textbf{277}
  (2008), no.~2, 439--458. \MR{MR2358291}

\bibitem[CGH08]{CGH08}
V.~Colin, E.~Giroux, and K.~Honda, \emph{Finitude homotopique et isotopique des
  structures de contact tendues}, preprint arXiv:math/0805.3051 (2008).

\bibitem[Che07]{Chel07}
I.~Cheltsov, \emph{{Fano varieties with many selfmaps}}, preprint;
  arXiv:math.AG/0611881 (2007).

\bibitem[CLPP05]{CLPP05}
M.~Cveti{\v{c}}, H.~L{\"u}, D.~N. Page, and C.~N. Pope, \emph{New
  {E}instein-{S}asaki spaces in five and higher dimensions}, Phys. Rev. Lett.
  \textbf{95} (2005), no.~7, 071101, 4. \MR{2167018}

\bibitem[Cua08]{Cua08}
J.~Cuadros, \emph{On null {S}asakian structures}, University of New Mexico
  Thesis (2008).

\bibitem[DK75]{DuKa75}
A.~Durfee and L.~Kauffman, \emph{Periodicity of branched cyclic covers}, Math.
  Ann. \textbf{218} (1975), no.~2, 157--174. \MR{52 \#6731}

\bibitem[EH01]{EtHo01}
John~B. Etnyre and Ko~Honda, \emph{On the nonexistence of tight contact
  structures}, Ann. of Math. (2) \textbf{153} (2001), no.~3, 749--766.
  \MR{MR1836287 (2002d:53119)}

\bibitem[Eli89]{Eli89}
Y.~Eliashberg, \emph{Classification of overtwisted contact structures on
  {$3$}-manifolds}, Invent. Math. \textbf{98} (1989), no.~3, 623--637.
  \MR{1022310 (90k:53064)}

\bibitem[Eli90]{Eli90}
Yakov Eliashberg, \emph{Filling by holomorphic discs and its applications},
  Geometry of low-dimensional manifolds, 2 (Durham, 1989), London Math. Soc.
  Lecture Note Ser., vol. 151, Cambridge Univ. Press, Cambridge, 1990,
  pp.~45--67. \MR{MR1171908 (93g:53060)}

\bibitem[Eli92]{Eli92}
Y.~Eliashberg, \emph{Contact {$3$}-manifolds twenty years since {J}.
  {M}artinet's work}, Ann. Inst. Fourier (Grenoble) \textbf{42} (1992),
  no.~1-2, 165--192. \MR{1162559 (93k:57029)}

\bibitem[FOW06]{FOW06}
A.~Futaki, H.~Ono, and G.~Wang, \emph{{Transverse {K}\"ahler geometry of
  {S}asaki manifolds and toric {S}asaki-{E}instein manifolds}}, preprint:
  arXiv:math.DG/0607586 (2006).

\bibitem[FS90]{FiSt90}
R.~Fintushel and R.~J. Stern, \emph{Instanton homology of {S}eifert fibred
  homology three spheres}, Proc. London Math. Soc. (3) \textbf{61} (1990),
  no.~1, 109--137. \MR{1051101 (91k:57029)}

\bibitem[Gei97]{Gei97}
H.~Geiges, \emph{Normal contact structures on {$3$}-manifolds}, Tohoku Math. J.
  (2) \textbf{49} (1997), no.~3, 415--422. \MR{98h:53046}

\bibitem[Gei08]{Gei08}
\bysame, \emph{An introduction to contact topology}, Cambridge Studies in
  Advanced Mathematics, vol. 109, Cambridge University Press, Cambridge, 2008.

\bibitem[Gir00]{Gir00}
Emmanuel Giroux, \emph{Structures de contact en dimension trois et bifurcations
  des feuilletages de surfaces}, Invent. Math. \textbf{141} (2000), no.~3,
  615--689. \MR{MR1779622 (2001i:53147)}

\bibitem[Gir01]{Gir01}
\bysame, \emph{Structures de contact sur les vari\'et\'es fibr\'ees en cercles
  audessus d'une surface}, Comment. Math. Helv. \textbf{76} (2001), no.~2,
  218--262. \MR{MR1839346 (2002c:53138)}

\bibitem[GK05]{GhKo05}
A~Ghigi and J.~Koll\'ar, \emph{{K\"ahler-{E}instein metrics on orbifolds and
  {E}instein metrics on Spheres}}, preprint; arXiv:math.DG/0507289, to appear
  in Comment. Math. Helvetici (2005).

\bibitem[GLS07]{GLS07b}
Paolo Ghiggini, Paolo Lisca, and Andr{\'a}s~I. Stipsicz, \emph{Tight contact
  structures on some small {S}eifert fibered 3-manifolds}, Amer. J. Math.
  \textbf{129} (2007), no.~5, 1403--1447. \MR{MR2354324}

\bibitem[GMSW04a]{GMSW04b}
J.~P. Gauntlett, D.~Martelli, J.~Sparks, and D.~Waldram, \emph{A new infinite
  class of {S}asaki-{E}instein manifolds}, Adv. Theor. Math. Phys. \textbf{8}
  (2004), no.~6, 987--1000. \MR{2194373}

\bibitem[GMSW04b]{GMSW04a}
\bysame, \emph{Sasaki-{E}instein metrics on {$S^2\times S^3$}}, Adv. Theor.
  Math. Phys. \textbf{8} (2004), no.~4, 711--734. \MR{2141499}

\bibitem[GMSY07]{GMSY06}
J.~P. Gauntlett, D.~Martelli, J.~Sparks, and S.-T. Yau, \emph{Obstructions to
  the existence of {S}asaki-{E}instein metrics}, Comm. Math. Phys. \textbf{273}
  (2007), no.~3, 803--827.

\bibitem[Gom08]{Gom08}
R.~Gomez, \emph{On {L}orenzian {S}asaki-{E}instein geometry}, University of New
  Mexico Thesis (2008).

\bibitem[Gro85]{Gro85}
M.~Gromov, \emph{Pseudoholomorphic curves in symplectic manifolds}, Invent.
  Math. \textbf{82} (1985), no.~2, 307--347. \MR{MR809718 (87j:53053)}

\bibitem[GS83]{GrSt83}
Gert-Martin Greuel and Joseph Steenbrink, \emph{On the topology of smoothable
  singularities}, Singularities, Part 1 (Arcata, Calif., 1981), Proc. Sympos.
  Pure Math., vol.~40, Amer. Math. Soc., Providence, R.I., 1983, pp.~535--545.
  \MR{MR713090 (84m:14006)}

\bibitem[GS03]{GhSc03}
Paolo Ghiggini and Stephan Sch{\"o}nenberger, \emph{On the classification of
  tight contact structures}, Topology and geometry of manifolds (Athens, GA,
  2001), Proc. Sympos. Pure Math., vol.~71, Amer. Math. Soc., Providence, RI,
  2003, pp.~121--151. \MR{MR2024633 (2004m:53146)}

\bibitem[Gui94]{Gui94b}
V.~Guillemin, \emph{Kaehler structures on toric varieties}, J. Differential
  Geom. \textbf{40} (1994), no.~2, 285--309. \MR{1293656 (95h:32029)}

\bibitem[HL75]{HaLa75}
F.~Reese Harvey and H.~Blaine Lawson, Jr., \emph{On boundaries of complex
  analytic varieties. {I}}, Ann. of Math. (2) \textbf{102} (1975), no.~2,
  223--290. \MR{MR0425173 (54 \#13130)}

\bibitem[Hon00a]{Hon00a}
Ko~Honda, \emph{On the classification of tight contact structures. {I}}, Geom.
  Topol. \textbf{4} (2000), 309--368 (electronic). \MR{MR1786111 (2001i:53148)}

\bibitem[Hon00b]{Hon00b}
\bysame, \emph{On the classification of tight contact structures. {II}}, J.
  Differential Geom. \textbf{55} (2000), no.~1, 83--143. \MR{MR1849027
  (2002g:53155)}

\bibitem[IF00]{Fle00}
A.~R. Iano-Fletcher, \emph{Working with weighted complete intersections},
  Explicit birational geometry of {$3$}-folds, London Math. Soc. Lecture Note
  Ser., vol. 281, Cambridge Univ. Press, Cambridge, 2000, pp.~101--173.
  \MR{2001k:14089}

\bibitem[JK01]{JoKo01b}
J.~M. Johnson and J.~Koll{\'a}r, \emph{Fano hypersurfaces in weighted
  projective 4-spaces}, Experiment. Math. \textbf{10} (2001), no.~1, 151--158.
  \MR{2002a:14048}

\bibitem[Kol04]{Kol04c}
J.~Koll\'ar, \emph{{Seifert $G_m$-bundles}}, preprint; arXiv:math.AG/0404386
  (2004).

\bibitem[Kol05]{Kol05b}
J.~Koll{\'a}r, \emph{Einstein metrics on five-dimensional {S}eifert bundles},
  J. Geom. Anal. \textbf{15} (2005), no.~3, 445--476. \MR{2190241}

\bibitem[Kol06]{Kol06a}
\bysame, \emph{Circle actions on simply connected 5-manifolds}, Topology
  \textbf{45} (2006), no.~3, 643--671. \MR{2218760}

\bibitem[Kol07]{Kol04}
J.~Koll\'ar, \emph{Einstein metrics on connected sums of {$S^2\times S^3$}}, J.
  Differential Geom. \textbf{77} (2007), no.~2, 259--272. \MR{2269781}

\bibitem[Kol08]{Kol06}
\bysame, \emph{{Positive {S}asakian structures on $5$-manifolds}}, Proceedings
  on the Conference on Riemannian Topology, to appear (K.~Galicki and
  S.~Simanca, eds.), Birkhauser, 2008.

\bibitem[Kur82]{Kur82}
Masatake Kuranishi, \emph{Strongly pseudoconvex {CR} structures over small
  balls. {III}. {A}n embedding theorem}, Ann. of Math. (2) \textbf{116} (1982),
  no.~2, 249--330. \MR{MR672837 (84h:32023c)}

\bibitem[Lee96]{Lee96}
J.~M. Lee, \emph{C{R} manifolds with noncompact connected automorphism groups},
  J. Geom. Anal. \textbf{6} (1996), no.~1, 79--90. \MR{1402387 (97h:32013)}

\bibitem[Ler02a]{Ler02a}
E.~Lerman, \emph{Contact toric manifolds}, J. Symplectic Geom. \textbf{1}
  (2002), no.~4, 785--828. \MR{2 039 164}

\bibitem[Ler02b]{Ler02}
\bysame, \emph{A convexity theorem for torus actions on contact manifolds},
  Illinois J. Math. \textbf{46} (2002), no.~1, 171--184. \MR{2003g:53146}

\bibitem[LM04]{LiMa04}
Paolo Lisca and Gordana Mati{\'c}, \emph{Transverse contact structures on
  {S}eifert 3-manifolds}, Algebr. Geom. Topol. \textbf{4} (2004), 1125--1144
  (electronic). \MR{MR2113899 (2006d:57041)}

\bibitem[LS07]{LiSt07}
P.~Lisca and A.~Stipsicz, \emph{On the existence of tight contact structures on
  {S}eifert fibered 3-manifolds}, preprint arXiv:math/0709.073 (2007).

\bibitem[Mas07]{Mas07}
P.~Massot, \emph{Geodesible contact structures on 3-manifolds}, preprint
  arXiv:math/0711.0377, to appear in Geometry and Topology (2007).

\bibitem[Mil68]{Mil68}
J.~Milnor, \emph{Singular points of complex hypersurfaces}, Annals of
  Mathematics Studies, No. 61, Princeton University Press, Princeton, N.J.,
  1968. \MR{39 \#969}

\bibitem[MO70]{MiOr70}
J.~Milnor and P.~Orlik, \emph{Isolated singularities defined by weighted
  homogeneous polynomials}, Topology \textbf{9} (1970), 385--393. \MR{45
  \#2757}

\bibitem[Mol88]{Mol88}
P.~Molino, \emph{Riemannian foliations}, Progress in Mathematics, vol.~73,
  Birkh\"auser Boston Inc., Boston, MA, 1988, Translated from the French by
  Grant Cairns, With appendices by Cairns, Y. Carri\`ere, \'E. Ghys, E. Salem
  and V. Sergiescu. \MR{89b:53054}

\bibitem[MS05]{MaSp05b}
D.~Martelli and J.~Sparks, \emph{Toric {S}asaki-{E}instein metrics on
  {$S^2\times S^3$}}, Phys. Lett. B \textbf{621} (2005), no.~1-2, 208--212.
  \MR{2152673}

\bibitem[MS06]{MaSp06}
\bysame, \emph{Toric geometry, {S}asaki-{E}instein manifolds and a new infinite
  class of {A}d{S}/{CFT} duals}, Comm. Math. Phys. \textbf{262} (2006), no.~1,
  51--89. \MR{2200882}

\bibitem[MSY06a]{MaSpYau05}
D.~Martelli, J.~Sparks, and S.-T. Yau, \emph{The geometric dual of
  {$a$}-maximisation for toric {S}asaki-{E}instein manifolds}, Comm. Math.
  Phys. \textbf{268} (2006), no.~1, 39--65. \MR{2249795}

\bibitem[MSY06b]{MaSpYau06}
\bysame, \emph{Sasaki-{E}instein manifolds and volume minimisation}, preprint;
  arXiv:hep-th/0603021 (2006).

\bibitem[MY07]{MaYe07}
G.~Marinescu and N.~Yeganefar, \emph{Embeddability of some strongly
  pseudoconvex {CR} manifolds}, Trans. Amer. Math. Soc. \textbf{359} (2007),
  no.~10, 4757--4771 (electronic). \MR{MR2320650}

\bibitem[NP07]{NiPa07}
K.~Niederkr\"uger and F.~Pasquotto, \emph{Resolution of symplectic cyclic
  orbifold singularities}, preprint; arXiv:math.SG/0707.4141 (2007).

\bibitem[Oh83]{Oh83}
H.~S. Oh, \emph{Toral actions on {$5$}-manifolds}, Trans. Amer. Math. Soc.
  \textbf{278} (1983), no.~1, 233--252. \MR{697072 (85b:57043)}

\bibitem[OR77]{OrRa77a}
P.~Orlik and R.~C. Randell, \emph{The monodromy of weighted homogeneous
  singularities}, Invent. Math. \textbf{39} (1977), no.~3, 199--211. \MR{57
  \#314}

\bibitem[Orl70]{Or70}
P.~Orlik, \emph{Weighted homogeneous polynomials and fundamental groups},
  Topology \textbf{9} (1970), 267--273. \MR{41 \#6251}

\bibitem[Orl72]{Or72}
\bysame, \emph{On the homology of weighted homogeneous manifolds}, Proceedings
  of the Second Conference on Compact Transformation Groups (Univ.
  Massachusetts, Amherst, Mass., 1971), Part I (Berlin), Springer, 1972,
  pp.~260--269. Lecture Notes in Math., Vol. 298. \MR{55 \#3312}

\bibitem[OS04]{OzSt04}
Burak Ozbagci and Andr{\'a}s~I. Stipsicz, \emph{Surgery on contact 3-manifolds
  and {S}tein surfaces}, Bolyai Society Mathematical Studies, vol.~13,
  Springer-Verlag, Berlin, 2004. \MR{MR2114165 (2005k:53171)}

\bibitem[PP07]{P-P07c}
Patrick Popescu-Pampu, \emph{On the cohomology rings of holomorphically
  fillable manifolds}, Math ArXivs: 0712.3484v1 (2007).

\bibitem[Ran75]{Ran75}
R.~C. Randell, \emph{The homology of generalized {B}rieskorn manifolds},
  Topology \textbf{14} (1975), no.~4, 347--355. \MR{54 \#1270}

\bibitem[Rei80]{Rei79}
M.~Reid, \emph{Canonical {$3$}-folds}, Journ\'ees de G\'eometrie Alg\'ebrique
  d'Angers, Juillet 1979/Algebraic Geometry, Angers, 1979, Sijthoff \&
  Noordhoff, Alphen aan den Rijn, 1980, pp.~273--310. \MR{82i:14025}

\bibitem[Sch95]{Sch95}
R.~Schoen, \emph{On the conformal and {CR} automorphism groups}, Geom. Funct.
  Anal. \textbf{5} (1995), no.~2, 464--481. \MR{1334876 (96h:53047)}

\bibitem[Sim96]{Sim96}
Santiago~R. Simanca, \emph{Precompactness of the {C}alabi energy}, Internat. J.
  Math. \textbf{7} (1996), no.~2, 245--254. \MR{MR1382725 (96m:58050)}

\bibitem[Sim00]{Sim00}
\bysame, \emph{Strongly extremal {K}\"ahler metrics}, Ann. Global Anal. Geom.
  \textbf{18} (2000), no.~1, 29--46. \MR{MR1739523 (2001b:58023)}

\bibitem[Sim04]{Sim04}
S.~R. Simanca, \emph{Canonical metrics on compact almost complex manifolds},
  Publica\c c\~oes Matem\'aticas do IMPA. [IMPA Mathematical Publications],
  Instituto de Matem\'atica Pura e Aplicada (IMPA), Rio de Janeiro, 2004.
  \MR{2113907}

\bibitem[Sma62]{Sm62}
S.~Smale, \emph{On the structure of {$5$}-manifolds}, Ann. of Math. (2)
  \textbf{75} (1962), 38--46. \MR{25 \#4544}

\bibitem[Spa08]{Spa08}
J.~Sparks, \emph{New results in {S}asaki-{E}instein geometry}, Proceedings on
  the Conference on Riemannian Topology, to appear (K.~Galicki and S.~Simanca,
  eds.), Birkhauser, 2008.

\bibitem[Tak78]{Tak}
T.~Takahashi, \emph{Deformations of {S}asakian structures and its application
  to the {B}rieskorn manifolds}, T\^ohoku Math. J. (2) \textbf{30} (1978),
  no.~1, 37--43. \MR{81e:53024}

\bibitem[Tan69]{Tan69b}
S.~Tanno, \emph{Sasakian manifolds with constant {$\phi$}-holomorphic sectional
  curvature}, T\^ohoku Math. J. (2) \textbf{21} (1969), 501--507. \MR{40
  \#4894}

\bibitem[Tia00]{Tia00}
G.~Tian, \emph{Canonical metrics in {K}\"ahler geometry}, Lectures in
  Mathematics ETH Z\"urich, Birkh\"auser Verlag, Basel, 2000, Notes taken by
  Meike Akveld. \MR{2001j:32024}

\bibitem[TZ02]{TiZh02}
Gang Tian and Xiaohua Zhu, \emph{A new holomorphic invariant and uniqueness of
  {K}\"ahler-{R}icci solitons}, Comment. Math. Helv. \textbf{77} (2002), no.~2,
  297--325. \MR{MR1915043 (2003i:32042)}

\bibitem[Var80]{Var80}
A.~N. Varchenko, \emph{Contact structures and isolated singularities}, Vestnik
  Moskov. Univ. Ser. I Mat. Mekh. (1980), no.~2, 18--21, 101. \MR{MR570582
  (81h:57015)}

\bibitem[vC06]{coe06}
C.~van Coevering, \emph{Toric surfaces and {S}asaki-{E}instein $5$-manifolds},
  SUNY at Stony Brook Ph.D. Thesis; see also preprint, arXiv:math.DG/0607721
  (2006).

\bibitem[vC07]{coe07}
\bysame, \emph{Some examples of toric {S}asaki-{E}instein manifolds}, preprint
  arXiv:math/0703501 (2007).

\bibitem[Web77]{Web77}
S.~M. Webster, \emph{On the transformation group of a real hypersurface},
  Trans. Amer. Math. Soc. \textbf{231} (1977), no.~1, 179--190. \MR{0481085 (58
  \#1231)}

\bibitem[WZ90]{WaZi90}
M.~Y. Wang and W.~Ziller, \emph{Einstein metrics on principal torus bundles},
  J. Differential Geom. \textbf{31} (1990), no.~1, 215--248. \MR{91f:53041}

\bibitem[WZ04]{WaZh04}
Xu-Jia Wang and Xiaohua Zhu, \emph{K\"ahler-{R}icci solitons on toric manifolds
  with positive first {C}hern class}, Adv. Math. \textbf{188} (2004), no.~1,
  87--103. \MR{MR2084775 (2005d:53074)}

\bibitem[Yam01]{Yam01}
T.~Yamazaki, \emph{On a surgery of {$K$}-contact manifolds}, Kodai Math. J.
  \textbf{24} (2001), no.~2, 214--225. \MR{2002c:57048}

\bibitem[YY05]{YaYu05}
S.~S.-T. Yau and Y.~Yu, \emph{Classification of {$3$}-dimensional isolated
  rational hypersurface singularities with {${\bf C}^\ast$}-action}, Rocky
  Mountain J. Math. \textbf{35} (2005), no.~5, 1795--1809. \MR{2206037
  (2006j:32034)}

\end{thebibliography}
\end{document}